\documentclass[11pt,twoside]{amsart}
\usepackage{amsmath,amssymb,amsthm}

\newtheorem{thm}{Theorem}[section]

 \newtheorem{cor}{Corollary}[section]
 \newtheorem{lem}{Lemma}[section]
 \newtheorem{prop}{Proposition}[section]
 \newtheorem{defn}{Definition}[section]
\newtheorem{rem}{Remark}[section]

\def\Id{{\rm Id}\,}
\def\d{\partial}
\def\ddj{\dot \Delta_j}

\def\ddk{\dot \Delta_k}

\def\tilde{\widetilde}

\def\wt{\widetilde}

\newcommand\R{\mathbb{R}}

\newcommand\Z{\mathbb{Z}}

\newcommand{\N}{\mathbb{N}}
\newcommand{\ep}{\varepsilon}

\newcommand{\curl}{\mbox{\rm curl}\;\!}
\renewcommand{\div}{\mbox{\rm div}\;\!}

\def\cA{{\mathcal A}}

\def\cC{{\mathcal C}}

\def\cP{{\mathcal P}}

\def\cX{{\mathcal X}}

\begin{document}
\title[Compressible Navier-Stokes equations]{Optimal decay for the compressible Navier-Stokes equations without additional smallness assumptions}
\author{Zhouping Xin}
\address{The Institute of Mathematical Sciences and Department of Mathematics, The Chinese University
of Hong Kong, Shatin, N.T., Hong Kong,}
\email{zpxin@ims.cuhk.edu.hk}

\author{Jiang Xu}
\address{Department of Mathematics,  Nanjing
University of Aeronautics and Astronautics,
Nanjing 211106, P.R.China,}
\email{jiangxu\underline{ }79math@yahoo.com}
\thanks{The second author would like to thank Professor R. Danchin for addressing the conjecture on the regularity of low frequencies when visiting the LAMA in UPEC}


\begin{abstract}
This work is concerned with the large time behavior of solutions to the barotropic compressible Navier-Stokes equations in $\mathbb{R}^{d}(d\geq2)$. Precisely, it is shown that if the initial density and velocity additionally belong to some Besov space $\dot{B}^{-\sigma_1}_{2,\infty}$ with $\sigma_1\in (1-d/2, 2d/p-d/2]$, then the $L^p$ norm (the slightly stronger $\dot{B}^{0}_{p,1}$ norm in fact) of global solutions admits the optimal decay $t^{-\frac{d}{2}(\frac 12-\frac 1p)-\frac{\sigma_1}{2}}$ for $t\rightarrow+\infty$. In contrast to refined time-weighted approaches (\cite{DX,XJ}), a pure energy argument
(\textit{independent of the spectral analysis}) has been developed in more general $L^p$ critical framework, which allows to \textit{remove the smallness of low frequencies of initial data}. Indeed, bounding the evolution of $\dot{B}^{-\sigma_1}_{2,\infty}$-norm restricted in low frequencies is the key ingredient, whose proof mainly depends on non standard $L^p$ product estimates with respect to different Sobolev embeddings. The result can hold true in case of \textit{large highly oscillating initial velocities}.
\end{abstract}
\maketitle

\section{Introduction}
Consider the following compressible Navier-Stokes equations
\begin{equation}\label{R-E1}
\left\{\begin{array}{l}
\partial_t\varrho+\div(\varrho u)=0,\\[1ex]
\partial_t(\varrho u)+\div(\varrho u\otimes u)-\div\bigl(2\mu D(u)+\lambda\,\div u\, \Id\bigr)+\nabla\Pi=0,
\end{array}\right.
\end{equation}
which govern the motion of a general barotropic compressible fluid in whole space  $\R^d$.  Here $u=u(t,x)\in \mathbb{R}^{d}$ (with  $(t,x)\in \mathbb{R}_{+}\times \mathbb{R}^{d}$) and $\varrho=\varrho(t,x)\in \mathbb{R}_{+}$ stand for the velocity field and density of the fluid, respectively.
The barotropic assumption means that the pressure $\Pi\triangleq P(\varrho)$ depends only upon the density of fluid and the function $P$ will be taken suitably smooth in what follows. The notation $D(u)\triangleq\frac12(D_x u+{}^T\!D_x u)$ stands for the {\it deformation tensor}, and
 $\div$ is the divergence operator with respect to the space variable.
The  density-dependent functions $\lambda$ and $\mu$  (the  \emph{bulk} and \emph{shear viscosities})
 are assumed to be smooth enough and to satisfy
 \begin{equation}\label{R-E2}
 \mu > 0\quad\hbox{and}\quad \nu\triangleq\lambda+2\mu>0.
 \end{equation}

System \eqref{R-E1}
is supplemented with the initial data
\begin{equation}\label{R-E3}
(\varrho,u)|_{t=0}=(\varrho_0,u_0).
\end{equation}
We investigate the solution $(\varrho,u)$ to the Cauchy problem \eqref{R-E1}-\eqref{R-E3} fulfilling the constant far-field behavior
$$(\varrho,u)\rightarrow (\varrho_\infty,0)$$
with $\varrho_\infty>0,$ as $|x|\rightarrow\infty$.

There is a huge literature on the existence, blow-up and large-time behavior of solutions to the compressible Navier-Stokes equations.
The local existence and uniqueness of smooth solutions away from vacuum were proved by Serrin \cite{S} and Nash \cite{N}. The local existence of strong solutions with Sobolev regularities was constructed by Solonnikov \cite{So}, Valli \cite{V} and Fiszdon-Zajaczkowski \cite{FZ}. Matsumura and Nishida \cite{MN1,MN2} established the global strong solutions for small perturbations of a linearly stable constant non-vacuum state $\left(\varrho _{\infty},0\right)$ in three dimensions.  With the additional $L^1$ \textit{smallness assumption} of initial data, the definite
decay rate was also available:
\begin{equation} \label{R-E4}\|(\varrho-\varrho_{\infty},u)(t)\|_{L^2}\lesssim \langle t\rangle^{-\frac{3}{4}}\quad\hbox{with }\
\langle t\rangle\triangleq \sqrt{1+t^2},\end{equation}
which coincides with that of the heat kernel. Indeed, the decay rate in time reveals the dissipative properties of solutions to \eqref{R-E1}-\eqref{R-E3}.
Subsequently, Ponce \cite{P} obtained more general $L^r$ decay rates:
\begin{equation}\label{R-E5}
\|\nabla^{k}(\varrho-\varrho_{\infty},u)(t)\|_{L^r}\lesssim \langle t\rangle^{-\frac{d}{2}(1-\frac{1}{r})-\frac k2}, \ \ \ 2\leq r\leq \infty, \ \ 0\leq k\leq 2,\ \ d=2,3.
\end{equation}
Later, Matsumura-Nishida's results were generalized to more physical situations where the fluid domain is not $\R^d$: for example, the exterior domains were investigated by Kobayashi \cite{K} and Kobayashi-Shibata \cite{KS}, the half space by Kagei \& Kobayashi \cite{KK1,KK2}. In addition, there are some results available which are connected to the information of wave propagation. Zeng \cite{Z} investigated the $L^1$ convergence to the nonlinear Burgers' diffusive wave. Hoff and Zumbrun \cite{HZ} performed the detailed analysis of the Green function for the multi-dimensional case and got the
$L^\infty$ decay rates of diffusion waves.  In \cite{LW}, Liu and Wang exhibited pointwise convergence of solution to diffusion waves with the optimal time-decay rate in odd dimension, where the phenomena of the weaker Huygens' principle was also shown.
Huang, Li and the first author \cite{HLX} established the global existence of classical solutions that may have large highly oscillations and
can contain vacuum states, see also \cite{LX} where further results including 2-dimensional case and large time behavior of solutions have been obtained. For the existence of solution with arbitrary data, a breakthrough is due to Lions \cite{L}, who obtained the
global existence of weak solutions with the finite energy when the adiabatic exponent is suitably large. Subsequently, some improvements were made by Feireisl, Novotny \& Petzeltov\'{a} \cite{FNP} and Jiang \& Zhang \cite{JZ}. In 1998, the first author \cite{X} found a fact that
any smooth solution to the Cauchy problem of full compressible Navier-Stokes system without heat conduction (including the current barotropic case) would blow up in finite time if the initial density contains vacuum. See also \cite{XY} for general blow-up results.

As regards global-in-time results, \textit{scaling invariance} plays a fundamental role. It is well known that suitable critical quantities (that is, having the same scaling invariance as the system under consideration) may control the possible blow-up of solutions. This trick is now classic. Let us recall
the existence context for the incompressible Navier-Stokes equations and go back to the work \cite{FK} by Fujita \& Kato (see also results by Kozono \& Yamazaki \cite{KY} and Cannone  \cite{C}). Observe that \eqref{R-E1} is invariant by the transform
\begin{equation}\label{R-E6}
\varrho(t,x)\leadsto \varrho(l ^2t,l x),\quad
u(t,x)\leadsto \ell u(l^2t,l x), \ \ l>0
\end{equation}
up to a change of the pressure term $\Pi$ into $l^2\Pi$. Danchin \cite{DR1} solved \eqref{R-E1} globally in critical homogeneous Besov spaces of $L^2$ type, that is, the initial data belong to $(\dot{B}^{\frac{d}{2}}_{2,1}\cap\dot{B}^{\frac{d}{2}-1}_{2,1})\times(\dot{B}^{\frac{d}{2}-1}_{2,1})^{d}$. Subsequently, the result of \cite{DR1} has been extended to the Besov spaces that are not related to $L^2$ by Charve \& Danchin \cite{CD} and Chen, Miao  \&  Zhang \cite{CMZ} independently. Haspot \cite{H} achieved essentially the same result by means of more elementary approach based on using the viscous effective flux as in \cite{HD,VK} (see also \cite{DH} for the case of density-dependent viscosity coefficients). We would like to give the statement as follows for convenience of reader.

\begin{thm}\label{thm1.1} (\cite{CD,CMZ,H})  Let $d\geq2$ and  $p$ satisfy
\begin{equation}\label{R-E7}
2\leq p \leq \min(4,d^*)\ \hbox{with}\  d^*\triangleq2d/(d-2)\ \hbox{and, additionally, }\  p\not=4\ \hbox{ if }\ d=2.\end{equation}
Assume  that $P'(\varrho_\infty)>0$ and that \eqref{R-E2} is fulfilled.  There exists a  constant $c=c(p,d,\lambda,\mu,P,\varrho_\infty)$ such that
if    $a_0\triangleq \varrho_0-\varrho_\infty$ is in $\dot B^{\frac d{p}}_{p,1},$
 if $u_0$ is in $\dot B^{\frac d{p}-1}_{p,1}$  and if in addition   $(a_0^\ell,u_0^\ell)\in\dot B^{\frac d2-1}_{2,1}$
   with
    \begin{equation}\label{R-E8}
\cX_{p,0}\triangleq \|(a_0,u_0)\|^\ell_{\dot B^{\frac d2-1}_{2,1}}+\|a_0\|^h_{\dot B^{\frac dp}_{p,1}}
+\|u_0\|^h_{\dot B^{\frac d{p}-1}_{p,1}}\leq c\end{equation}
then \eqref{R-E1}  has a unique global-in-time  solution $(\varrho,u)$  with  $\varrho=\varrho_\infty+a$ and
$(a,u)$ in the space $X_p$ defined by\footnote{The subspace $\,\wt\cC_b(\R_+;\dot B^{s}_{q,1})\,$ of $\,\cC_b(\R_+;\dot B^{s}_{q,1})\,$
is defined in \eqref{R-E177}, and the norms $\|\cdot\|_{\wt L^\infty(\dot B^s_{p,1})}$
are introduced just below Definition \ref{Defn3.2}.}:
$$\displaylines{
(a,u)^\ell\in \wt\cC_b(\R_+;\dot B^{\frac d2-1}_{2,1})\cap  L^1(\R_+;\dot B^{\frac d2+1}_{2,1}),\quad
a^h\in \wt\cC_b(\R_+;\dot B^{\frac dp}_{p,1})\cap L^1(\R_+;\dot B^{\frac dp}_{p,1}),
\cr u^h\in  \wt\cC_b(\R_+;\dot B^{\frac dp-1}_{p,1})
\cap L^1(\R_+;\dot B^{\frac dp+1}_{p,1}).}
$$
Furthermore, there exists some constant $C=C(p,d,\lambda,\mu,P,\varrho_\infty)$ such that
\begin{equation}\label{R-E9}
\cX_p\leq C \cX_{p,0},
\end{equation} with
\begin{multline}\label{R-E10}
 \cX_{p}\triangleq\|(a,u)\|^{\ell}_{\wt L^\infty(\dot B^{\frac d2-1}_{2,1})}+\|(a,u)\|^{\ell}_{L^1(\dot B^{\frac d2+1}_{2,1})}
\\+\|a\|^{h}_{\wt L^\infty(\dot B^{\frac dp}_{p,1})\cap L^1(\dot B^{\frac dp}_{p,1})}
+\|u\|^{h}_{\wt L^\infty(\dot B^{\frac dp-1}_{p,1})\cap L^1(\dot B^{\frac dp+1}_{p,1})}.
\end{multline}
\end{thm}

Next, a natural problem is what is the large time asymptotic behavior of global-in-time solutions constructed above. Although providing an accurate long-time asymptotics description is still out of reach in general, there are a number of works concerning time-decay rates of $L^q$-$L^r$ (same as \eqref{R-E4}-\eqref{R-E5}) in the
critical regularity framework. Okita \cite{O} established the optimal time-decay estimates to \eqref{R-E1}-\eqref{R-E3} by using a smart modification of the method of \cite{DR1}. However, that result cannot cover the 2D case. In the survey \cite{DR2}, Danchin proposed another description of the
time decay which allows one to handle any dimensions $d\geq2$. Danchin and the second author \cite{DX} developed the method of \cite{DR2} so as to get the optimal decay rates in more general $L^p$ critical spaces. The regularity exponent $d/p-1$ for velocity may become negative in physical dimensions $d=2,3$, the result thus applies to \textit{large highly oscillating initial velocities}. Inspired by the private communication with Danchin, the second author \cite{XJ}
claimed a general low-frequency assumption for optimal decay estimates, that is, the initial density and velocity additionally belong to some Besov space $\dot{B}^{-\sigma_1}_{2,\infty}$ where the regularity exponent $\sigma_1$ belongs to a whole range $(1-d/2, \sigma_0]$ with $\sigma_{0}\triangleq \frac{2d}{p}-\frac d2$. These recent efforts (also including classical works \cite{MN2,P}) all depend on the time-weighted energy approach in \textit{the Fourier semi-group framework}, so the smallness of low frequencies of initial data is usually needed. In the present paper, we aim at developing a \textit{different} $L^p$ energy argument under the same regularity assumption as in \cite{XJ}, and then establish the optimal decay of solutions to \eqref{R-E1}-\eqref{R-E3}. In comparison with previous studies (whether in Sobolev spaces with higher regularity or critical Besov spaces), \textit{the smallness condition} could be removed.

\section{Reformulation and main results} \setcounter{equation}{0}
To state main results, it is convenient to rewrite System \eqref{R-E1} as the
linearized compressible Navier-Stokes equations about equilibrium $(\varrho_{\infty},0)$, and regard the nonlinearities as source terms.
Precisely, one has
\begin{equation}\label{R-E11}
\left\{\begin{array}{l}\d_ta+\div u=f,\\[1ex]
\d_tu-\cA u+\nabla a=g,
\end{array}\right.
\end{equation}
with $f\triangleq-\div(au),\,$
$\cA\triangleq\mu_\infty\Delta+(\lambda_\infty\!+\!\mu_\infty)\nabla\div$ such that $\mu_\infty>0$ and $\lambda_\infty+2\mu_\infty>0,$
$$  g\triangleq-u\cdot\nabla u-I(a)\cA u-k(a)\nabla a+\frac1{1+a}\div\bigl(2\wt\mu(a) D(u)+\wt\lambda(a)\div u\:\Id\bigr),
 $$
where\footnote{In our analysis, the exact value of functions $k,$ $\tilde{\lambda},$ $\tilde{\mu}$ and even $I$ will not  matter :
 we  only need those functions to be smooth enough and to vanish at $0$.}
  $$
  I(a)\triangleq\frac{a}{1+a},\quad\!\! k(a)\triangleq\frac{P'(1+a)}{1+a}-1,\quad\!\!
\wt\mu(a)\triangleq\mu(1+a)-\mu(1)\ \hbox{  and  }\  \wt\lambda(a)\triangleq\lambda(1+a)-\lambda(1).$$
For simplicity, one can
perform a suitable  rescaling  in \eqref{R-E1} so as to normalize, at infinity, the density  $\varrho_\infty,$ the sound speed $c_\infty\triangleq\sqrt{P'(\varrho_\infty)}$
and the total viscosity $\nu_\infty\triangleq \lambda_\infty+2\mu_\infty$ (with
$\lambda_{\infty}\triangleq\lambda(\varrho_\infty)$ and $\mu_{\infty}\triangleq\mu(\varrho_\infty)$)
to be one.

Denote $\Lambda^{s}f\triangleq \mathcal{F}^{-1}(|\xi|^{s}\mathcal{F}f)$ for $s\in \mathbb{R}$. Now, main results are stated as follows.
\begin{thm}\label{thm1.2}
Let those assumptions of Theorem \ref{thm1.1} hold and ~$(\varrho,u)$ be the corresponding global solution to \eqref{R-E1}. If in addition $(a_0,u_0)^{\ell}\in \dot{B}^{-\sigma_1}_{2,\infty}$ ($1-\frac{d}{2}<\sigma_1\leq\sigma_{0}\triangleq \frac{2d}{p}-\frac d2$) such that
$\|(a_0,u_0)\|^{\ell}_{\dot{B}^{-\sigma_{1}}_{2,\infty}}$ is bounded, then it holds that
\begin{eqnarray}\label{R-E12}
\|\Lambda^{\sigma}(a,u)\|_{L^p}\lesssim (1+t)^{-\frac{d}{2}(\frac 12-\frac 1p)-\frac{\sigma+\sigma_1}{2}} \ \mbox{if}\quad -\tilde{\sigma}_{1}<\sigma\leq \frac dp-1,
\end{eqnarray}
for all $t\geq0$, where $\tilde{\sigma}_{1}\triangleq \sigma_{1}+d(\frac 12-\frac 1p)$.
\end{thm}

Furthermore, by using improved Gagliardo-Nirenberg inequalities, one can obtain the following optimal decay estimates of $\dot{B}^{-\sigma_{1}}_{2,\infty}$-$L^r$ type.
\begin{cor}\label{cor1.1}
Let those assumptions of Theorem \ref{thm1.2} be fulfilled. Then the corresponding solution $(a,u)$ admits the following decay estimates
\begin{eqnarray}\label{R-E13}
\|\Lambda^{l}(a,u)\|_{L^r}\lesssim \bigl (1+t)^{-\frac d2(\frac 12-\frac 1r)-\frac {l+\sigma_{1}}{2}},
\end{eqnarray}
for $p\leq r\leq\infty$, $l\in\R$ and
$-\tilde{\sigma}_1<l+d\big(\frac1p-\frac1r\big)\leq \frac dp-1\cdotp$
\end{cor}

Some comments are in order.
\begin{enumerate}
\item In comparison with previous efforts, the new ingredient is that
\textit{the smallness of low frequencies of initial data} is no longer needed for the time-decay properties of solutions in Theorem \ref{thm1.2} and Corollary \ref{cor1.1}. Furthermore, choosing the endpoint regularity for example $\sigma_1\equiv\sigma_0=d/2$ (if $p=2$) allows to go back to the classical decay estimates \eqref{R-E4}-\eqref{R-E5} with aid of the Sobolev embedding $L^1\hookrightarrow\dot{B}^{-d/2}_{2,\infty}$. Clearly, there is some freedom in the choice  of $\sigma_1$, which enables us to obtain more optimal decay estimates in the $L^p$ framework.
    \smallbreak
\item In \cite{DX} (setting $\sigma_1=\sigma_0$), there is a little loss on those decay rates due to different Sobolev embeddings at low frequencies and high frequencies. In fact, it was shown that the solution itself decayed to equilibrium in $L^p$ norm with the rate of $O\Big(t^{-d(1/p-1/4)}\Big)$ for $t\rightarrow\infty$. The present results avoid this minor flaw and indicate the optimal decay as fast as $t^{-d/2p}$, which are satisfactory.
 \smallbreak
\item In the energy argument with interpolation, the optimality of lower bound of $\sigma_1$ can be confirmed by non standard $L^p$
 product estimates in Corollary \ref{cor5.1} and a couple of interpolation inequalities like \eqref{R-E82}, \eqref{R-E98} and \eqref{R-E102}.
\smallbreak
\item In physical dimensions $d=2,3,$  Condition \eqref{R-E7} allows to consider the case  $p>d$ for which the
velocity regularity exponent $d/p-1$  becomes negative.
The result thus applies to \emph{large} highly oscillating initial velocities (see \cite{CD,CMZ} for more explanations).
\smallbreak
 \item  Finally, it is worth pointing out our approach is of \textit{independent interest} in the $L^p$ critical framework, which is likely to
 effective for other systems that are encountered in fluid dynamics.
\end{enumerate}
\medbreak

In what follows, let us give some illustration on Theorem \ref{thm1.2}. Throughout the proof,
our task is to establish a Lyapunov-type inequality in time for energy norms (see (\ref{R-E99})) by using the pure energy argument (independent of spectral analysis). In fact, the idea comes from the previous works \cite{GW,SG}. In Sobolev spaces with higher regularity, they deduced the Lyapunov-type inequality with respect to the time variable by using a family of scaled energy estimates with minimum derivative counts and interpolations among them, which leads to desired time-decay estimates for several Boltzmann type equations and compressible Navier-Stokes equations. In the $L^p$ critical regularity framework however, one cannot afford any loss of regularity for the high frequency part of the solution (and some terms as $u\cdot\nabla a$ induce a loss of one derivative since there is no smoothing for $a$, solution of a transport equation), so their approach fails to take effect in critical spaces.

Precisely, the proof is divided into several lemmas. Lemma \ref{lem4.1} is devoted to the linear low-frequency estimate. As in \cite{DR1}, we decompose the velocity into the ``incompresible part $\omega$" and the ``compressible part $v$". Then it suffices to study the mixed system between $a$ and $v$, since $\omega$ satisfies, up to nonlinear terms, a mere heat equation. Hyperbolic energy methods enable us to obtain the parabolic decay for the low frequencies of $(a,u)$.
Lemma \ref{lem4.2} is dedicated to the linear high-frequency estimate. In order to eliminate the major difficulty that there is a loss of one derivative of density, we employ the $L^p$ energy approach in terms of \textit{effective velocity} $w$ (given by \eqref{R-E29}) that was first used in the critical framework by Haspot \cite{H}. The interested reader is also referred to \cite{HD,VK} for the use of viscous effective flux. The dissipative mechanism of \eqref{R-E1}-\eqref{R-E2} is well exhibited by Lemmas \ref{lem4.1}-\ref{lem4.2}. Next, let us highlight new ingredients in this paper. Lemma \ref{lem5.1} is responsible for the evolution of Besov norm $\dot{B}^{-\sigma_{1}}_{2,\infty}$ (restricted in the low-frequency part) of solutions. Due to the general regularity that $\sigma_1$ belongs to $(1-d/2,\sigma_0]$, the proof of Lemma \ref{lem5.1} is quite technical. More precisely, by using low and high frequency decompositions,
one splits the nonlinear term $(f,g)$ into $(f^\ell,g^\ell)$ and $(f^h,g^h)$ (see the context below). To deal with $\|(f,g)^{\ell}\|^{\ell}_{\dot{B}^{-\sigma_{1}}_{2,\infty}}$, some \textit{new and non standard} Besov product estimates (see Proposition \ref{prop5.1}) are developed by means of Bony's para-product decomposition, which can be regarded as the supplement of Proposition \ref{prop3.5}. Indeed, the inequalities \eqref{R-E53}-\eqref{R-E54} in Corollary \ref{cor5.1} are enough for us to derive desired nonlinear estimates. On the other hand, bounding
$\|(f,g)^{h}\|^{\ell}_{\dot{B}^{-\sigma_{1}}_{2,\infty}}$ seems to be more elaborate.
The non oscillation case $(2\leq p\leq d$) and the oscillation case $(p>d$) are handled separately by using different Sobolev embeddings and Besov product estimates. Combining these analysis leads to the evolution of Besov norm $\dot{B}^{-\sigma_{1}}_{2,\infty}$ of solutions. Finally, nonlinear product estimates (in the $L^p$ framework) and real interpolations lead to the Lyapunov-type inequality \eqref{R-E99} for energy norms.

We end the section with a simple overview of main results. In comparison with \cite{DX,XJ}, a different decay framework has been
 established by employing $L^p$ energy methods (independent of spectral analysis), which allows to remove
 \textit{the smallness assumption of low frequencies of initial data}. Of course, vacuum is ruled out in this paper. In the presence of vacuum, the corresponding mathematical theory for viscous fluids is still far away well known in critical spaces, which is left in the forthcoming work.

The rest of the paper unfolds as follows: In Section 3, we recall briefly the Littlewood-Paley decomposition, Besov spaces and related analysis tools.
In Section 4, we establish the low-frequency and high-frequency estimates of solutions by using the pure energy arguments. Section 5 is devoted to bounding the evolution of negative Besov norms, which plays the key role in deriving the Lyapunov-type inequality for energy norms. In the last section (Section 6), we present the proofs of Theorem \ref{thm1.2} and Corollary \ref{cor1.1}.

\section{Preliminary}\setcounter{equation}{0}
Throughout the paper, $C>0$ stands for a generic harmless ``constant''. For brevity, we sometime write
$u\lesssim v$ instead of $u\leq Cv$. The notation $u\approx v$ means that $%
u\lesssim v$ and $v\lesssim u$. For any Banach space $X$ and $u,v\in X$,
$\left\|\left(u,v\right)\right\| _{X}\triangleq \left\|u\right\| _{X}+\left\|v\right\|_{X}$. For
all $T>0$ and $\rho\in\left[1,+\infty\right]$, one denotes by
$L_{T}^{\rho}(X) \triangleq L^{\rho}\left(\left[0,T\right];X\right)$ the set of measurable functions $u:\left[0,T\right]\rightarrow X$ such that $t\mapsto\left\|u(t)\right\|_{X}$ is in $L^{\rho}\left(0,T\right)$.

\subsection{Littlewood-Paley decomposition and Besov spaces}
Let us briefly recall the Littlewood-Paley decomposition and Besov spaces for convenience. More details
may be found for example in Chap. 2 and Chap. 3 of \cite{BCD}. Choose a smooth
radial non increasing function $\chi $ with $\mathrm{Supp}\,\chi \subset
B\left(0,\frac {4}{3}\right)$ and $\chi \equiv 1$ on $B\left(0,\frac
{3}{4}\right)$. Set $\varphi \left(\xi\right) =\chi \left(\xi/2\right)-\chi \left(\xi\right)$. It is not difficult to check that
$$
\sum_{j\in \mathbb{Z}}\varphi \left( 2^{-j}\cdot \right) =1\ \ \text{in}\ \
\mathbb{R}^{d}\setminus \{ 0\} \ \ \text{and}\ \ \mathrm{Supp}\,\varphi \subset \left\{ \xi \in \mathbb{R}^{d}:3/4\leq |\xi|\leq 8/3\right\} .
$$
Then homogeneous dyadic blocks $\dot{\Delta}_j(j\in \mathbb{Z})$ are defined by
$$
\dot{\Delta}_{j}u\triangleq \varphi (2^{-j}D)u=\mathcal{F}^{-1}(\varphi
(2^{-j}\cdot )\mathcal{F}u)=2^{jd}h(2^{j}\cdot )\star u\ \ \hbox{with}\ \
h\triangleq \mathcal{F}^{-1}\varphi .
$$
Consequently, one has the unit decomposition for any tempered distribution $u\in S^{\prime }(\mathbb{R}^{d})$
\begin{equation}\label{R-E14}
u=\sum_{j\in \mathbb{Z}}\dot{\Delta}_{j}u.
\end{equation}
As it holds only modulo polynomials, it is convenient to consider the subspace of those tempered distributions $u$ such that
\begin{equation}\label{R-E15}
\lim_{j\rightarrow -\infty }\| \dot{S}_{j}u\| _{L^{\infty} }=0,
\end{equation}
where $\dot{S}_{j}f$ stands for the low frequency cut-off defined by $\dot{S}_{j}u\triangleq\chi \left(2^{-j}D\right)u$. Indeed,
if \eqref{R-E15} is fulfilled, then \eqref{R-E14} holds in $S'\left(\mathbb{R}^{d}\right)$. Hence, we denote by $S'_{0}\left(\mathbb{R}^{d}\right)$ the subspace of tempered distributions satisfying \eqref{R-E15}.

Based on those dyadic blocks,  Besov spaces are defined as follows.
\begin{defn}\label{Defn3.1}
For $s\in \mathbb{R}$ and $1\leq p,r\leq\infty,$ the homogeneous
Besov spaces $\dot{B}^{s}_{p,r}$ are defined by
$$\dot{B}^{s}_{p,r}\triangleq\left\{u\in S'_{0}:\left\|u\right\|_{\dot{B}^{s}_{p,r}}<+\infty\right\},$$
where
\begin{equation}\label{R-E16}
\left\|u\right\|_{\dot B^{s}_{p,r}}\triangleq\|(2^{js}\|\dot{\Delta}_{j}u\|_{L^p})\|_{\ell^{r}(\mathbb{Z})}.
\end{equation}
\end{defn}
Moreover, a class of mixed space-time Besov spaces are also used when studying evolutionary PDEs, which was introduced by J.-Y. Chemin and N. Lerner \cite{CL} (see also \cite{CJY} for the particular case of Sobolev spaces).
\begin{defn}\label{Defn3.2}
 For $T>0, s\in\mathbb{R}, 1\leq r,\theta\leq\infty$, the homogeneous Chemin-Lerner space $\widetilde{L}^{\theta}_{T}(\dot{B}^{s}_{p,r})$
is defined by
$$\widetilde{L}^{\theta}_{T}(\dot{B}^{s}_{p,r})\triangleq\left\{u\in L^{\theta}\left(0,T;S'_{0}\right):\left\|u\right\|_{\widetilde{L}^{\theta}_{T}(\dot{B}^{s}_{p,r})}<+\infty\right\},$$
where
\begin{equation}\label{R-E17}
\|u\|_{\widetilde{L}^{\theta}_{T}(\dot{B}^{s}_{p,r})}\triangleq\|(2^{js}\| \dot{\Delta}_{j}u\|_{L^{\theta}_{T}(L^{p})})\|_{\ell^{r}(\mathbb{Z})}.
\end{equation}
\end{defn}
For notational simplicity, index $T$ is omitted if $T=+\infty $.
We also use the following functional space:
\begin{equation}\label{R-E177}
\tilde{\mathcal{C}}_{b}(\mathbb{R_{+}};\dot{B}_{p,r}^{s})\triangleq \left\{u \in
\mathcal{C}(\mathbb{R_{+}};\dot{B}_{p,r}^{s})\ \hbox{s.t}\ \left\|u\right\| _{\tilde{L}^{\infty}(\dot{B}_{p,r}^{s})}<+\infty \right\} .
\end{equation}
Thanks to the Minkowski's inequality, one has the following topological relation between Chemin-Lerner's spaces and the standard mixed spaces
$L_{T}^{\theta} (\dot{B}_{p,r}^{s})$.
\begin{rem}\label{Rem3.1}
It holds that
$$\left\|u\right\|_{\widetilde{L}^{\theta}_{T}(B^{s}_{p,r})}\leq\left\|u\right\|_{L^{\theta}_{T}(B^{s}_{p,r})} \ \
\mbox{if} \ \ r\geq\theta;\ \ \ \
\left\|u\right\|_{\widetilde{L}^{\theta}_{T}(B^{s}_{p,r})}\geq\left\|u\right\|_{L^{\theta}_{T}(B^{s}_{p,r})} \ \
\mbox{if} \ \ r\leq\theta.
$$
\end{rem}
Restricting the above norms \eqref{R-E16} and \eqref{R-E17} to the low or high
frequencies parts of distributions will be crucial in our approach. For example, let us fix some integer $j_{0}$ (the value of which will
follow from the proof of the main theorem) and set\footnote{Note that for technical reasons, we need a small
overlap between low and high frequencies.}
$$\left\| u\right\| _{\dot{B}_{p,1}^{s}}^{\ell} \triangleq \sum_{j\leq
j_{0}}2^{js}\| \dot{\Delta}_{j}f\|_{L^{p}} \ \ \mbox{and} \ \ \|u\|_{\dot{B}_{p,1}^{s}}^{h}\triangleq \sum_{j\geq j_{0}-1}2^{js}\| \dot{\Delta}_{j}u\| _{L^{p}},$$
$$\left\|u\right\| _{\tilde{L}_{T}^{\infty} (\dot{B}_{p,1}^{s})}^{\ell} \triangleq
\sum_{j\leq j_{0}}2^{js}\|\dot{\Delta}_{j}u\|_{L_{T}^{\infty} (L^{p})} \ \
\mbox{and} \ \ \|u\| _{\tilde{L}_{T}^{\infty} (\dot{B}_{p,1}^{s})}^{h}\triangleq \sum_{j\geq j_{0}-1}2^{js}\| \dot{\Delta}_{j}u\|
_{L_{T}^{\infty} (L^{p})}.$$

\subsection{Analysis tools in Besov spaces}
Recall the classical \emph{Bernstein inequality}:
\begin{equation}\label{R-E18}
\|D^{k}u\|_{L^{b}}
\leq C^{1+k} \lambda^{k+d(\frac{1}{a}-\frac{1}{b})}\left\|u\right\|_{L^{a}}
\end{equation}
that holds for all function $u$ such that $\mathrm{Supp}\,\mathcal{F}u\subset\left\{\xi\in \R^{d}: \left|\xi\right|\leq R\lambda\right\}$ for some $R>0$
and $\lambda>0$, if $k\in\N$ and $1\leq a\leq b\leq\infty$.

More generally, if $u$ satisfies $\mathrm{Supp}\,\mathcal{F}u\subset \left\{\xi\in \R^{d}:
R_{1}\lambda\leq\left|\xi\right|\leq R_{2}\lambda\right\}$ for some $0<R_{1}<R_{2}$  and $\lambda>0$,
then for any smooth  homogeneous of degree $m$ function $A$ on $\R^d\setminus\{0\}$ and $1\leq a\leq\infty,$ it holds that
(see e.g. Lemma 2.2 in \cite{BCD}):
\begin{equation}\label{R-E19}
\left\|A(D)u\right\|_{L^{a}}\lesssim\lambda^{m}\left\|u\right\|_{L^{a}}.
\end{equation}
An obvious  consequence of \eqref{R-E18} and \eqref{R-E19} is that
$\left\|D^{k}u\right\|_{\dot{B}^{s}_{p, r}}\thickapprox\left \|u\right\|_{\dot{B}^{s+k}_{p, r}}$ for all $k\in\N.$

The following nonlinear generalization of \eqref{R-E19} will be also used (see Lemma 8 in \cite{D1}):
\begin{prop}\label{prop3.1}
If $\mathrm{Supp}\,\mathcal{F}f\subset \{\xi\in \mathbb{R}^{d}:
R_{1}\lambda\leq|\xi|\leq R_{2}\lambda\}$ then there exists $c$ depending only on $d,$ $R_1$  and $R_2$
so that for all $1<p<\infty,$
$$c\lambda^2\biggl(\frac{p-1}p\biggr)\int_{\R^d}|f|^p\,dx\leq (p-1)\int_{\R^d}|\nabla f|^2|f|^{p-2}\,dx
=-\int_{\R^d}\Delta f\: |f|^{p-2}f\,dx.
$$
\end{prop}

The following classical properties are often used (see \cite{BCD}):
\begin{prop}\label{prop3.2}
\begin{itemize}
\item \ \emph{Scaling invariance:} For any $s\in \mathbb{R}$ and $(p,r)\in
[1,\infty ]^{2}$, there exists a constant $C=C(s,p,r,d)$ such that for all $\lambda >0$ and $u\in \dot{B}_{p,r}^{s}$, then
$$
C^{-1}\lambda ^{s-\frac {d}{p}}\left\|u\right\|_{\dot{B}_{p,r}^{s}}
\leq \left\|u(\lambda \cdot)\right\|_{\dot{B}_{p,r}^{s}}\leq C\lambda ^{s-\frac {d}{p}}\left\|u\right\|_{\dot{B}_{p,r}^{s}}.
$$

\item \ \emph{Completeness:} $\dot{B}^{s}_{p,r}$ is a Banach space whenever $%
s<\frac{d}{p}$ or $s\leq \frac{d}{p}$ and $r=1$.

\item \ \emph{Action of Fourier multipliers:} If $F$ is a smooth homogeneous of
degree $m$ function on $\mathbb{R}^{d}\backslash \{0\}$, then
$$F(D):\dot{B}_{p,r}^{s}\rightarrow \dot{B}_{p,r}^{s-m}.$$
\end{itemize}
\end{prop}

\begin{prop}\label{prop3.3}
Let $1\leq p, r,r_{1},r_{2}\leq \infty$.
\begin{itemize}
\item  Complex interpolation: if $u\in \dot{B}_{p,r_1}^{s}\cap \dot{B}_{p,r_2}^{\tilde{s}} $ and $s\neq \tilde{s}$, then $u\in \dot{B}_{p,r}^{\theta s+(1-\theta)\tilde{s}}$ for all $\theta\in (0,1)$ and
$$\left\|u\right\|_{\dot{B}_{p,r}^{\theta s+(1-\theta )\tilde{s}}}\leq \left\|u\right\|_{\dot{B}_{p,r_{1}}^{s}}^{\theta} \left\|u\right\|_{\dot{B}_{p,r_2}^{\tilde{s}}}^{1-\theta }$$
with $\frac{1}{r}=\frac{\theta}{r_{1}}+\frac{1-\theta}{r_{2}}$.

\item Real interpolation: if $u\in \dot{B}_{p,\infty}^{s}\cap \dot{B}_{p,\infty}^{\tilde{s}} $ and $s<\tilde{s}$, then $u\in \dot{B}_{p,1}^{\theta s+(1-\theta)\tilde{s}}$ for all $\theta\in (0,1)$ and
$$\left\|u\right\|_{\dot{B}_{p,1}^{\theta s+(1-\theta )\tilde{s}}}\leq \frac{C}{\theta(1-\theta)(\tilde{s}-s)} \left\|u\right\| _{\dot{B}_{p,\infty}^{s}}^{\theta} \left\|u\right\|_{\dot{B}_{p,\infty}^{\tilde{s}}}^{1-\theta }.$$
\end{itemize}
\end{prop}
The following embedding is also used frequently throughout this paper.
\begin{prop} \label{prop3.4} (Embedding for Besov spaces on $\mathbb{R}^{d}$)
\begin{itemize}
\item For any $p\in[1,\infty]$ we have the  continuous embedding
$\dot {B}^{0}_{p,1}\hookrightarrow L^{p}\hookrightarrow \dot {B}^{0}_{p,\infty}.$
\item If $\sigma\in\R$, $1\leq p_{1}\leq p_{2}\leq\infty$ and $1\leq r_{1}\leq r_{2}\leq\infty,$
then $\dot {B}^{\sigma}_{p_1,r_1}\hookrightarrow
\dot {B}^{\sigma-d(\frac{1}{p_{1}}-\frac{1}{p_{2}})}_{p_{2},r_{2}}$.
\item The space  $\dot {B}^{\frac {d}{p}}_{p,1}$ is continuously embedded in the set  of
bounded  continuous functions \emph{(}going to zero at infinity if, additionally, $p<\infty$\emph{)}.
\end{itemize}
\end{prop}

The following product estimates in Besov spaces play a key role in our analysis of the bilinear terms of \eqref{R-E11} (see \cite{BCD,DX}).
\begin{prop}\label{prop3.5}
Let $s>0$ and $1\leq p,r\leq\infty$. Then $\dot{B}^{s}_{p,r}\cap L^{\infty}$ is an algebra and
\begin{equation*}
\left\|uv\right\|_{\dot{B}^{s}_{p,r}}\lesssim \left\|u\right\|_{L^{\infty}}\left\|v\right\|_{\dot{B}^{s}_{p,r}}+\left\|v\right\|_{L^{\infty}}\left\|u\right\|_{\dot{B}^{s}_{p,r}}.
\end{equation*}
Let the real numbers $s_{1},$ $s_{2},$ $p_1$  and $p_2$ be such that
\begin{equation*}
s_{1}+s_{2}>0,\quad s_{1}\leq\frac {d}{p_{1}},\quad s_{2}\leq\frac {d}{p_{2}},\quad
s_{1}\geq s_{2},\quad\frac{1}{p_{1}}+\frac{1}{p_{2}}\leq1.
\end{equation*}
Then it holds that
\begin{equation*}
\left\|uv\right\|_{\dot{B}^{s_{2}}_{q,1}}\lesssim \left\|u\right\|_{\dot{B}^{s_{1}}_{p_{1},1}}\left\|v\right\|_{\dot{B}^{s_{2}}_{p_{2},1}}\quad\hbox{with}\quad
\frac1{q}=\frac1{p_{1}}+\frac1{p_{2}}-\frac{s_{1}}d\cdotp
\end{equation*}
Additionally, for exponents $s>0$ and $1\leq p_{1},p_{2},q\leq\infty$ satisfying
\begin{equation*}
\frac{d}{p_{1}}+\frac{d}{p_{2}}-d\leq s \leq\min\left(\frac {d}{p_{1}},\frac {d}{p_{2}}\right)\quad\hbox{and}\quad \frac{1}{q}=\frac {1}{p_{1}}+\frac {1}{p_{2}}-\frac{s}{d},
\end{equation*}
one has
\begin{equation*}
\left\|uv\right\|_{\dot{B}^{-s}_{q,\infty}}\lesssim\left\|u\right\|_{\dot{B}^{s}_{p_{1},1}}\left\|v\right\|_{\dot{B}^{-s}_{p_{2},\infty}}.
\end{equation*}
\end{prop}

System \eqref{R-E11}  also involves compositions of functions (through $I(a)$, $k(a)$, $\tilde{\lambda}(a)$ and $\tilde{\mu}(a)$) and they
are bounded according to the following conclusion (see \cite{BCD,DX}).
\begin{prop}\label{prop3.6}
Let $F:\R\rightarrow\R$ be smooth with $F(0)=0$.
For  all  $1\leq p,r\leq\infty$ and $s>0$, it holds that
$F(u)\in \dot {B}^{s}_{p,r}\cap L^{\infty}$ for $u\in \dot {B}^{s}_{p,r}\cap L^{\infty}$,  and
$$\left\|F(u)\right\|_{\dot B^{s}_{p,r}}\leq C\left\|u\right\|_{\dot B^{s}_{p,r}}$$
with $C$ depending only on $\left\|u\right\|_{L^{\infty}}$, $F'$ \emph{(}and higher derivatives\emph{)}, $s$, $p$ and $d$.

In the case $s>-\min\left(\frac {d}{p},\frac {d}{p'}\right)$ then $u\in\dot {B}^{s}_{p,r}\cap\dot {B}^{\frac {d}{p}}_{p,1}$
implies that $F(u)\in \dot {B}^{s}_{p,r}\cap\dot {B}^{\frac {d}{p}}_{p,1}$, and
\begin{equation*}
\left\|F(u)\right\|_{\dot B^{s}_{p,r}}\leq C(1+\|u\|_{\dot {B}^{\frac {d}{p}}_{p,1}})\|u\|_{\dot {B}^{s}_{p,r}}.
\end{equation*}
\end{prop}

The following  commutator estimates (see \cite{DX}) have been used  in the high-frequency estimate of the proof of Theorem \ref{thm1.2}.
\begin{prop}\label{prop3.7}
Let $1\leq p,\,p_{1}\leq\infty$ and
\begin{equation*}
-\min\left(\frac{d}{p_{1}},\frac{d}{p'}\right)<s\leq 1+\min\left(\frac {d}{p},\frac{d}{p_{1}}\right) \ \ \ \mbox{with}\ \ \ \frac{1}{p}+\frac{1}{p'}=1.
\end{equation*}
There exists a constant $C>0$ depending only on $s$ such that for all $j\in\Z$ and $\ell\in\left\{1,\cdots,d\right\}$, it holds that
\begin{equation*}
\|[v\cdot\nabla,\d_\ell\dot{\Delta}_{j}]a\|_{L^{p}}\leq
Cc_{j}2^{-j\left(s-1\right)}\left\|\nabla v\right\|_{\dot{B}^{\frac{d}{p_{1}}}_{p_{1},1}}\left\|\nabla a\right\|_{\dot{B}^{s-1}_{p,1}},
\end{equation*}
where the commutator
$\left[\cdot,\cdot\right]$ is defined by $\left[f,g\right]=fg-gf$ and $(c_{j})_{j\in\Z}$ denotes
a sequence such that $\left\|(c_{j})\right\|_{\ell^{1}}\leq 1$.
\end{prop}

Finally, we end this section with those inequalities of Gagliardo-Nirenberg type,
which can be found in the work of Sohinger and Strain \cite{SS} (see also Chap. 2 of \cite{BCD} or \cite{XK2}).
\begin{prop}\label{prop3.8}
The following interpolation inequalities hold true:
\begin{eqnarray*}
\|\Lambda^{\ell}f\|_{L^{r}}\lesssim \|\Lambda^{m}f\|^{1-\theta}_{L^{q}}\|\Lambda^{k}f\|^{\theta}_{L^{q}},
\end{eqnarray*}
whenever  $0\leq\theta\leq1$,  $1\leq q\leq r\leq\infty$ and $$\ell+d\Big(\frac{1}{q}-\frac{1}{r}\Big)=m(1-\theta)+k\theta.$$
\end{prop}

%*****************************************

\section{Low-frequency and high-frequency analysis}\setcounter{equation}{0}
We are going to establish a Lyapunov-type inequality for energy norms by using a pure energy argument.
For clarity, the proof is divided into several steps. In this section, we derive first the low-frequency and high-frequency estimates.

\subsection{Low-frequency estimates}
Let $\Lambda^{s}z\triangleq\mathcal{F}^{-1}\left(\left|\xi\right|^{s}\mathcal{F}z\right)$ ($s\in \mathbb{R}$). Denote by $\omega=\Lambda^{-1}\curl u$ the incompressible part of $u$ and by $v=\Lambda^{-1}\div u$ the compressible part of $u$. Therefore, we see that $\omega$ satisfies the heat equation
\begin{equation}\label{R-E20}
\partial_{t}\omega-\mu_{\infty}\Delta\omega=G_1, \ \ \ \ G_1\triangleq\Lambda^{-1}\curl g.
\end{equation}
It is not difficult to deduce that
\begin{eqnarray}\label{R-E200}
\frac{d}{dt}\|\omega^{\ell}\|_{\dot{B}^{\frac{d}{p}-1}_{p,1}}+\|\omega^{\ell}\|_{\dot{B}^{\frac{d}{p}+1}_{p,1}}\lesssim \|g^{\ell}\|_{\dot{B}^{\frac{d}{p}-1}_{p,1}}
\end{eqnarray}
for $1\leq p\leq \infty$ and  $t\geq0$, where $z^{\ell}\triangleq S_{k_0}z$.  On the other hand, one can study the coupled system
for $(a,v)$:
\begin{equation}\label{R-E21}
\left\{\begin{array}{l}\d_ta+\Lambda v=f,\\[1ex]
\d_tv-\Delta v-\Lambda a=h,\ \ \ \ h\triangleq\Lambda^{-1}\div g.
\end{array}\right.
\end{equation}
\begin{lem}\label{lem4.1}
Let $k_0$ be some integer. Then it holds that
\begin{equation}\label{R-E22}
\frac{d}{dt}\|(a,u)^{\ell}\|_{\dot{B}^{\frac{d}{2}-1}_{2,1}}+\|(a,u)^{\ell}\|_{\dot{B}^{\frac{d}{2}+1}_{2,1}}\lesssim \|(f,g)^{\ell}\|_{\dot{B}^{\frac{d}{2}-1}_{2,1}},
 \end{equation}
for all $t\geq0$.
\end{lem}
\begin{proof}
Set $z_k\triangleq \dot{\Delta}_{k}z$. Apply the operator $\dot{\Delta}_{k}S_{k_0}$ to (\ref{R-E21}). By using the standard energy argument, we arrive at the following three equalities:
\begin{equation}\label{R-E23}
\frac 12 \frac{d}{dt}(\|a_{k}^{\ell}\|^2_{L^2}+\|v_{k}^{\ell}\|^2_{L^2})+\|\Lambda v_{k}^{\ell}\|^2_{L^2}=(f_{k}^{\ell}|a_{k}^{\ell})+(h_{k}^{\ell}|v_{k}^{\ell}),
\end{equation}
\begin{equation}\label{R-E24}
-\frac 12\frac{d}{dt}(v_{k}^{\ell}|\Lambda a_{k}^{\ell})+\|\Lambda a_{k}^{\ell}\|^2_{L^2}=-(\Delta v_{k}^{\ell}|\Lambda a_{k}^{\ell})+\|\Lambda v_{k}^{\ell}\|^2_{L^2}-(h_{k}^{\ell}|\Lambda a_{k}^{\ell})-(v_{k}^{\ell}|\Lambda f_{k}^{\ell})
\end{equation}
and
\begin{equation}\label{R-E25}
\frac 12\frac{d}{dt}\|\Lambda a_{k}^{\ell}\|^2_{L^2}=(\Delta v_{k}^{\ell}|\Lambda a_{k}^{\ell})+(f_{k}^{\ell}|\Lambda^2 a_{k}^{\ell}),
\end{equation}
from which one can deduce that
\begin{multline}\label{R-E26}
\frac 12 \frac{d}{dt}\mathcal{L}^2_{k}(t)+\|(\Lambda a_{k}^{\ell}, \Lambda v_{k}^{\ell})\|^2_{L^2}\hfill \cr \hfill =2(f_{k}^{\ell}|a_{k}^{\ell})+2(h_{k}^{\ell}|v_{k}^{\ell})-(h_{k}^{\ell}|\Lambda a_{k}^{\ell})-(v_{k}^{\ell}|\Lambda f_{k}^{\ell})+(f_{k}^{\ell}|\Lambda^2 a_{k}^{\ell})
\end{multline}
with $\mathcal{L}^2_{k}\triangleq 2(\|a_{k}^{\ell}\|^2_{L^2}+\|v_{k}^{\ell}\|^2_{L^2})+\|\Lambda a_{k}^{\ell}\|^2_{L^2}-2(v_{k}^{\ell}|\Lambda a_{k}^{\ell})$. It follows from Young's inequality that
$\mathcal{L}^2_{k}\approx \|(a_{k}^{\ell}, \Lambda a_{k}^{\ell},v_{k}^{\ell})\|^2_{L^2}\approx \|(a_{k}^{\ell},v_{k}^{\ell})\|^2_{L^2}$, due to the low-frequency cut-off. Consequently, we get the following inequality
\begin{equation}\label{R-E27}
\frac 12 \frac{d}{dt}\mathcal{L}^2_{k}+2^{2k}\mathcal{L}^2_{k}\lesssim \|(f_{k}^{\ell},h_{k}^{\ell})\|_{L^2}\mathcal{L}_{k},
\end{equation} which leads to
\begin{equation} \label{R-E28}
\frac{d}{dt}\mathcal{L}_{k}+2^{2k}\mathcal{L}_{k}\lesssim\|(f_{k}^{\ell},h_{k}^{\ell})\|_{L^2}.
\end{equation}
Then, multiplying both sides by $2^{k(d/2-1)}$ and summing up on $k\in\mathbb{Z}$ gives
\begin{equation}\label{R-E288}
\frac{d}{dt}\|(a,v)^{\ell}\|_{\dot{B}^{\frac{d}{2}-1}_{2,1}}+\|(a,v)^{\ell}\|_{\dot{B}^{\frac{d}{2}+1}_{2,1}}\lesssim \|(f,g)^{\ell}\|_{\dot{B}^{\frac{d}{2}-1}_{2,1}},
 \end{equation}
Hence, \eqref{R-E22} is followed by \eqref{R-E200} and \eqref{R-E288} directly.
\end{proof}

\subsection{High-frequency estimates}
In the high-frequency regime, the problem is that the structure of $f$ would cause a loss of one derivative as there is no smoothing effect
for $a$. To overcome the difficulty, as in \cite{H}, we perform the energy argument
in terms of the \emph{effective velocity}
 \begin{equation}\label{R-E29}
 w\triangleq\nabla(-\Delta)^{-1}(a-\div u).\end{equation}
Note  that if \eqref{R-E11} is written in terms of $a,$ $w$ and of the divergence free part $\cP u$
 of $u,$ then, up to low order terms, $a$ satisfies a \emph{damped} transport equation, and
 both $w$ and $\cP u$ satisfy a constant heat equation. Thanks to this structure of the system, one can establish the high-frequency estimates.

 \begin{lem} \label{lem4.2}Let $k_{0}\in \mathbb{Z}$ be chosen suitably large. It holds that for all $t\geq0$,
 \begin{eqnarray}\label{R-E30}
 &&\frac{d}{dt}\|(\nabla a, u)\|_{\dot B^{\frac dp-1}_{p,1}}^h+(\|\nabla a\|_{\dot B^{\frac dp-1}_{p,1}}^h+\|u\|_{\dot B^{\frac dp+1}_{p,1}}^h)\nonumber\\ & \lesssim &\|f\|_{\dot B^{\frac dp-2}_{p,1}}^h+\|g\|_{\dot B^{\frac dp-1}_{p,1}}^h+\|\nabla u\|_{\dot B^{\frac dp}_{p,1}}\|a\|_{\dot B^{\frac dp}_{p,1}},
 \end{eqnarray}
where $$\|z\|^{h}_{\dot{B}^{s}_{2,1}}\triangleq \sum_{k\geq k_0-1}2^{ks}\|\dot{\Delta}_{k}z\|_{L^2}\quad \mbox{for} \ s\in \mathbb{R}.$$
 \end{lem}
 \begin{proof}
Note that $\cP u$ satisfies
  $$
\d_t\cP u -\mu_\infty\Delta\cP u=\cP g.
$$
Applying $\ddk$ to the above equation yields for all $k\in\Z,$
$$
\d_t\cP u_k -\mu_\infty\Delta\cP u_k=\cP g_k\quad\hbox{with }\ u_k\triangleq \ddk u\ \hbox{ and }\
g_k\triangleq \ddk g.
$$
Then, multiplying each component of the above equation
by $|(\cP u_k)^i|^{p-2}(\cP u_k)^i$ and  integrating over $\R^d$ gives  for $i=1,\cdots,d,$
$$\frac 1p\frac d{dt}\|\cP u_k^i\|_{L^p}^p-\mu_\infty\int \Delta(\cP u_k)^i|(\cP u_k)^i|^{p-2}(\cP u_k)^i\,dx
=\int|(\cP u_k)^i|^{p-2}(\cP u_k)^i\cP g^i_k\,dx.$$
The key observation is that the second term of the l.h.s., although not spectrally localized,
may be bounded from below as if it were (see Proposition \ref{prop3.1}).
After summation  on $i=1,\cdots,d,$ we end up  (for some constant $c_p$ depending only $p$) with
\begin{equation}\label{R-E31}
\frac1p\frac d{dt}\|\cP u_k\|_{L^p}^p+c_p\mu_\infty2^{2k}\|\cP u_k\|_{L^p}^p
\leq \|\cP g_k\|_{L^p}\|\cP u_k\|_{L^p}^{p-1},
\end{equation}
which leads to
\begin{equation}\label{R-E32}
\frac d{dt}\|\cP u_k\|_{L^p}+c_p\mu_\infty2^{2k}\|\cP u_k\|_{L^p}
\leq \|\cP g_k\|_{L^p}.
\end{equation}
On the other hand, it is clear that $w$ fulfills
\begin{eqnarray}\label{R-E33}
\d_tw-\Delta w=\nabla(-\Delta)^{-1}(f-\div g) +w-(-\Delta)^{-1}\nabla a.
\end{eqnarray}
Hence, arguing exactly as for the proof of \eqref{R-E31} shows that for $w_k\triangleq\ddk w$:
\begin{multline}\label{R-E34}
\frac1p\frac d{dt}\|w_k\|_{L^p}^p+c_p2^{2k}\|w_k\|_{L^p}^p\\
\leq  \bigl(\|\nabla(-\Delta)^{-1}(f_k-\div g_k)\|_{L^p}
+\|w_k-(-\Delta)^{-1}\nabla a_k\|_{L^p}  \bigr)\|w_k\|_{L^p}^{p-1}.
\end{multline}
Similarly, we obtain
\begin{multline} \label{R-E35}
\frac d{dt}\|w_k\|_{L^p}+c_{p}2^{2k}\|w_k\|_{L^p}\\ \leq \|\nabla(-\Delta)^{-1}(f_{k}-\div g_{k})\|_{L^p}+\|w_{k}-(-\Delta)^{-1}\nabla a_{k}\|_{L^p}.
\end{multline}
In terms of $w,$ the function $a$ satisfies
the following \emph{damped} transport equation:
\begin{equation}\label{R-E36}
 \d_ta+\div(au)+a=-\div w.
 \end{equation}
Then, applying the operator $\d_i\dot{\Delta}_{k}$  to \eqref{R-E36}
and denoting $R_k^i\triangleq[u\cdot\nabla,\d_i\ddk]a$
yields
\begin{equation}\label{R-E37}
 \d_t\d_ia_k+u\cdot\nabla\d_ia_k+\d_ia_k=   -\d_i\ddk(a\div u)-\d_i\div w_k+R_k^i,\quad i=1,\cdots,d.\end{equation}
Multiplying  by $|\d_ia_k|^{p-2}\d_ia_k,$   integrating on $\R^d,$ and performing an integration by
parts in the second term of  \eqref{R-E37}, one can get
$$\displaylines{
\frac1p\frac{d}{dt}\|\d_ia_k\|_{L^p}^p+\|\d_ia_k\|_{L^p}^p=\frac1p\int\div u\:|\d_ia_k|^p\,dx
\hfill\cr\hfill+\int\bigl(R_k^i -\d_i\ddk(a\div u)-\d_i\div w_k)|\d_ia_k|^{p-2}\d_ia_k\,dx.}
$$
Summing up on $i=1,\cdots,d$ and applying H\"older  and Bernstein inequalities give
\begin{multline}\label{R-E38}
\frac1p\frac{d}{dt}\|\nabla a_k\|_{L^p}^p+\|\nabla a_k\|_{L^p}^p\leq\Bigl(\frac{1}{p}\|\mathrm{div}u\|_{L^\infty}\|\nabla a_k\|_{L^p}+\|\nabla\ddk(a\mathrm{div}u)\|_{L^p}\\+C2^{2k}\|w_k\|_{L^p}+\|R_k\|_{L^p}\Bigr)\|\nabla a_k\|_{L^p}^{p-1},
\end{multline}
which leads to
\begin{multline} \label{R-E39}
\frac d{dt}\|\nabla a_k\|_{L^p}+\|\nabla a_k\|_{L^p}\\ \leq\Big(\frac{1}{p}\|\div u\|_{L^\infty}\|\nabla a_k\|_{L^p}+
\|\nabla\dot{\Delta}_{k}(a\div u)\|_{L^p}+C2^{2k}\|w_{k}\|_{L^p}+\|R_{k}\|_{L^p}\Big).
\end{multline}
Furthermore, adding up (\ref{R-E39}) (multiplying by $\gamma c_p$ for some $\gamma>0$) to \eqref{R-E32} and \eqref{R-E35} yields
$$\displaylines{
\frac{d}{dt}\bigl(\|\cP u_k\|_{L^p}+\|w_k\|_{L^p}+\gamma c_p\|\nabla a_k\|_{L^p}\bigr)
+c_p2^{2k}\bigl(\mu_\infty\|\cP u_k\|_{L^p}+\|w_k\|_{L^p})\hfill\cr\hfill
+\gamma c_p\|\nabla a_k\|_{L^p}\leq \|\cP g_k\|_{L^p}+\|\nabla(-\Delta)^{-1}(f_k-\div g_k)\|_{L^p}\hfill\cr\hfill
+\gamma c_p\Bigl(\frac{1}{p}\|\mathrm{div}u\|_{L^\infty}\|\nabla a_k\|_{L^p}+\|\nabla\ddk(a\div u)\|_{L^p}+\|R_k\|_{L^p}\Bigr)\hfill\cr\hfill
+C\gamma c_p2^{2k}\|w_k\|_{L^p}
+\|w_k-(-\Delta)^{-1}\nabla a_k\|_{L^p}.}
$$
Noticing that $(-\Delta)^{-1}$ is a homogeneous Fourier multiplier of degree $-2,$ we have
$$\|(-\Delta)^{-1}\nabla a_k\|_{L^p}\lesssim 2^{-2k}\|\nabla a_{k}\|_{L^p}\lesssim 2^{-2k_0}\|\nabla a_{k}\|_{L^p}
\quad\hbox{for all }\  k\geq k_{0}-1.$$
Choosing $ k_{0}$ suitably large and $\gamma$ sufficiently small, we deduce that there exists a constant $c_0>0$ such that for all $k\geq k_{0}-1$,
$$\displaylines{
\frac{d}{dt}\bigl(\|\cP u_k\|_{L^p}+\|w_k\|_{L^p}+\|\nabla a_k\|_{L^p}\bigr)
+c_0\bigl(2^{2k}\|\cP u_k\|_{L^p}+2^{2k}\|w_k\|_{L^p}+\|\nabla a_k\|_{L^p})\hfill\cr\hfill
\lesssim  2^{-k}\|f_k\|_{L^p}+\|g_k\|_{L^p}\hfill\cr\hfill
+\Bigl(\|\mathrm{div}u\|_{L^\infty}\|\nabla a_k\|_{L^p}+\|\nabla\ddk(a\div u)\|_{L^p}+\|R_k\|_{L^p}\Bigr).}
$$
Since
\begin{eqnarray}\label{R-E40}
 u=w-\nabla(-\Delta)^{-1}a+\cP u,
\end{eqnarray}
it follows that
\begin{eqnarray} \label{R-E41}
&&\frac{d}{dt}\|(\nabla a_k,u_k)\|_{L^p}+c_0\|(\nabla a_k,2^{2k}u_k)\|_{L^p}\nonumber\\&\lesssim&\|(2^{-k}f_k,g_k)\|_{L^p}+\|\mathrm{div}u\|_{L^\infty}\|\nabla a_k\|_{L^p}+\|\nabla\ddk(a\div u)\|_{L^p}+\|R_k\|_{L^p}.
\end{eqnarray}
Hence, by multiplying both sides \eqref{R-E41} by $2^{k(\frac{d}{p}-1)}$ and summing up over $k\geq k_0-1$, we arrive at (\ref{R-E30}).
\end{proof}

%***********************************************

\section{The evolution of negative Besov norm}\setcounter{equation}{0}
In this section, we establish the evolution of the negative Besov norms at low frequencies, which plays the key role in deriving the
Lyapunov-type inequality for energy norms. To this end, we need some non classical product estimates as follows.
\begin{prop}\label{prop5.1}
Let the real numbers $s_{1},$ $s_{2},$ $p_1$  and $p_2$ be such that
$$
s_1+s_2\geq0,\quad s_1\leq\frac d{p_1},\quad s_2<\min\Big(\frac d{p_1},\frac d{p_2}\Big)\quad \mbox{and}\quad \frac1{p_1}+\frac1{p_2}\leq1.
$$
Then it holds that
\begin{eqnarray}\label{R-E42}
\|fg\|_{\dot{B}^{s_{1}+s_{2}-\frac{d}{p_1}}_{p_2,\infty}}\lesssim \|f\|_{\dot{B}^{s_{1}}_{p_1,1}}\|g\|_{\dot{B}^{s_{2}}_{p_2,\infty}}.
\end{eqnarray}
\end{prop}
\begin{proof}
Here, we need the so-called Bony decomposition  for the product of two tempered
distributions $f$ and $g$ (see for example \cite{BCD}):
\begin{equation}\label{R-E43}
fg=T_fg+R(f,g)+T_{g}f,
\end{equation}
where  the \emph{paraproduct} between  $f$ and $g$ is defined by
$$T_fg:=\sum_j \dot S_{j-1} f\ddj g\quad\hbox{with }\ \dot S_{j-1}\triangleq\chi(2^{-(j-1)}D),$$
and  the \emph{remainder} $R(f,g)$ is given by the series:
$$
R(f,g):=\sum_j \ddj f\, (\dot\Delta_{j-1}g+\ddj g+\dot\Delta_{j+1}g\bigr).
$$
It follows from the definition of $T_gf$ and the spectral cut-off property that
\begin{eqnarray}\label{R-E44}
\dot{\Delta}_{j}T_gf=\dot{\Delta}_{j}\Big(\sum_{j'} \dot S_{j'-1} g \dot{\Delta}_{j'} f\Big)=\sum_{|j-j'|\leq 4}\dot{\Delta}_{j}(\dot S_{j'-1} g\dot{\Delta}_{j'} f)
\end{eqnarray}
where $\dot S_{j'-1} g$ can be given by $\sum_{k\leq j'-2}\ddk g$. Hence, it follows from the H\"{o}lder inequality and Bernstein inequality \eqref{R-E18} that
\begin{multline}\label{R-E45}
\|\dot{\Delta}_{j}T_gf\|_{L^{p_2}} \lesssim \sum_{|j-j'|\leq 4}\|\dot S_{j'-1} g\dot{\Delta}_{j'}f\|_{L^{p_2}}\hfill \cr \hfill \lesssim
\left\{\begin{array}{l}
\sum_{|j-j'|\leq 4}\sum_{k\leq j'-2}\|\ddk g\|_{L^m}\|\dot{\Delta}_{j'}f\|_{L^{p_1}}, \quad \frac{1}{p_{2}}=\frac 1m+\frac{1}{p_{1}}\ ( p_{2}\leq p_{1});\\[1ex]
\sum_{|j-j'|\leq 4}\sum_{k\leq j'-2}\|\ddk g\|_{L^\infty}\|\dot{\Delta}_{j'}f\|_{L^{p_2}}, \quad p_{2}> p_{1}
\end{array}\right.
\hfill \cr \hfill \lesssim
\left\{\begin{array}{l}
\sum_{|j-j'|\leq 4}\sum_{k\leq j'-2}2^{k(\frac{d}{p_1}-s_{2})} 2^{ks_{2}}\|\ddk g\|_{L^{p_2}}\|\dot{\Delta}_{j'}f\|_{L^{p_1}}, \quad p_{2}\leq p_{1};\\[1ex]
\sum_{|j-j'|\leq 4}\sum_{k\leq j'-2}2^{j'd(\frac{1}{p_1}-\frac{1}{p_2})} 2^{k(\frac{d}{p_2}-s_{2})}2^{ks_{2}}\|\ddk g\|_{L^{p_2}}\|\dot{\Delta}_{j'}f\|_{L^{p_1}}, \quad p_{2}> p_{1}.
\end{array}\right.
\end{multline}
Furthermore, the right side of \eqref{R-E45} can be bounded by
$$\sum_{|j-j'|\leq 4} 2^{j'(\frac{d}{p_1}-s_{1}-s_{2})} 2^{j's_{1}}\|\dot{\Delta}_{j'}f\|_{L^{p_1}}\|g\|_{\dot{B}^{s_{2}}_{p_2,\infty}},
$$
provided $s_2<\min(d/{p_1}, d/{p_2})$. Consequently,
\begin{eqnarray}\label{R-E46}
\|T_gf\|_{\dot{B}^{s_{1}+s_{2}-\frac{d}{p_1}}_{p_2,\infty}}\lesssim \|f\|_{\dot{B}^{s_{1}}_{p_1,1}}\|g\|_{\dot{B}^{s_{2}}_{p_2,\infty}}.
\end{eqnarray}

\noindent \underline{Case 1: $s_{1}\geq0$}.
We use the fact that $T$ maps $L^{q_1}\times \dot{B}^{s_{2}}_{p_2,\infty}$
to $\dot{B}^{s_{2}}_{q,\infty}$, that is,
\begin{equation}\label{R-E48}
\|T_fg\|_{\dot{B}^{s_{2}}_{q,\infty}}\lesssim  \|f\|_{L^{q_1}}\|g\|_{\dot{B}^{s_{2}}_{p_2,\infty}} \quad \quad \mbox{with}\quad \frac 1q=\frac {1}{q_{1}}+\frac {1}{p_{2}},
\end{equation}
and the embedding $\dot{B}^{s_{1}}_{p_1,1}\hookrightarrow L^{q_{1}}$ with $\frac {1}{q_{1}}=\frac {1}{p_{1}}-\frac {s_{1}}{d}$ to get
\begin{equation}\label{R-E49}
\|T_fg\|_{\dot{B}^{s_{2}}_{q,\infty}} \lesssim \|f\|_{\dot{B}^{s_{1}}_{p_1,1}}\|g\|_{\dot{B}^{s_{2}}_{p_2,\infty}}.
\end{equation}
Furthermore, we use the embedding $\dot{B}^{s_{2}}_{q,\infty}\hookrightarrow\dot{B}^{s_{1}+s_{2}-\frac{d}{p_1}}_{p_2,\infty}$ (if $s_{1}\leq \frac{d}{p_{1}}$ ) and arrive at
\begin{equation}\label{R-E50}
\|T_fg\|_{\dot{B}^{s_{1}+s_{2}-\frac{d}{p_1}}_{p_2,\infty}} \lesssim \|f\|_{\dot{B}^{s_{1}}_{p_1,1}}\|g\|_{\dot{B}^{s_{2}}_{p_2,\infty}}.
\end{equation}
\noindent \underline{Case 2: $s_{1}<0$}. In this case, the fact that $T$ maps $\dot{B}^{s_{2}}_{p_{2},\infty}\times \dot{B}^{s_{1}}_{p_{1},\infty}$
to $\dot{B}^{s_{1}+s_{2}}_{\tilde{p},\infty}$ enables us to obtain
\begin{equation}\label{R-E51}
\|T_fg\|_{\dot{B}^{s_{1}+s_{2}}_{\tilde{p},\infty}}\lesssim\|f\|_{\dot{B}^{s_{1}}_{p_{1},1}}\|g\|_{\dot{B}^{s_{2}}_{p_{2},\infty}}
\end{equation}
with $1/\tilde{p}=1/p_{1}+1/p_{2}$. Noticing that the embedding $\dot{B}^{s_{1}+s_{2}}_{\tilde{p},\infty}\hookrightarrow\dot{B}^{s_{1}+s_{2}-\frac{d}{p_1}}_{p_2,\infty}$, we arrive at \begin{equation}\label{R-E52}
\|T_fg\|_{\dot{B}^{s_{1}+s_{2}-\frac{d}{p_1}}_{p_2,\infty}}\lesssim \|f\|_{\dot{B}^{s_{1}}_{p_{1},1}}\|g\|_{\dot{B}^{s_{2}}_{p_{2},\infty}}.
\end{equation}

To bound the term $R(f,g),$ one can use that the remainder operator $R$ maps
$\dot B^{s_1}_{p_1,1}\times\dot B^{s_2}_{p_2,\infty}$ to $\dot B^{s_1+s_2}_{\tilde{p},\infty}$
with $1/\tilde{p}=1/p_{1}+1/p_{2}$, where $s_1+s_2\geq 0$ and $\frac1{p_1}+\frac1{p_2}\leq 1$. Then
the embedding $\dot B^{s_1+s_2}_{\tilde{p},\infty}\hookrightarrow\dot{B}^{s_{1}+s_{2}-\frac{d}{p_1}}_{p_2,\infty}$
yields the desired inequality.

Therefore, the proof of Proposition \ref{prop5.1} is finished.
\end{proof}

\begin{cor}\label{cor5.1}
Let $1-\frac d2<\sigma_{1}\leq \sigma_{0}$ and $p$ satisfy \eqref{R-E7}. The following two inequalities hold true:
\begin{equation}\label{R-E53}
\|fg\|_{\dot{B}^{-\sigma_{1}}_{2,\infty}}\lesssim \|f\|_{\dot{B}^{\frac dp}_{p,1}}\|g\|_{\dot{B}^{-\sigma_{1}}_{2,\infty}},
\end{equation}
and
\begin{equation}\label{R-E54}
\|fg\|_{\dot{B}^{\frac dp-\frac d2-\sigma_{1}}_{2,\infty}}\lesssim \|f\|_{\dot{B}^{\frac dp-1}_{p,1}}\|g\|_{\dot{B}^{\frac dp-\frac d2-\sigma_{1}+1}_{2,\infty}}.
\end{equation}
\end{cor}
\begin{proof}
Observe that $\sigma_1\leq\sigma_0\leq \frac dp$ and $-\sigma_1<\frac d2-1\leq \frac dp$ (owing to $2\leq p\leq d^{*}$). Hence,
\eqref{R-E53} stems from Proposition \ref{prop5.1} with $s_{1}=\frac dp, s_{2}=-\sigma_{1}, p_{1}=p$ and $p_{2}=2$. Also,
\eqref{R-E54} is followed by Proposition \ref{prop5.1} with $s_{1}=\frac {d}{p}-1, s_{2}=\frac dp-\frac d2-\sigma_{1}+1$ and $p_{1}=p$ and $p_{2}=2$.
\end{proof}

Now, we turn to bound the evolution of negative Besov norm, which is the main ingredient in the proof of Theorem \ref{thm1.2}.
\begin{lem}\label{lem5.1}Let $1-\frac d2<\sigma_{1}\leq \sigma_{0}$ and $p$ satisfy \eqref{R-E7}.
It holds that
\begin{multline}\label{R-E55}
\Big(\|(a,u)(t)\|^{\ell}_{\dot B^{-\sigma_{1}}_{2,\infty}}\Big)^2\lesssim \Big(\|(a_0,u_0)\|^{\ell}_{\dot B^{-\sigma_{1}}_{2,\infty}}\Big)^2\hfill \cr \hfill+\int^{t}_{0}\Big(D^{1}_{p}(\tau)+D^{2}_{p}(\tau)\Big)\Big(\|(a,u)(\tau)\|^{\ell}_{\dot B^{-\sigma_{1}}_{2,\infty}}\Big)^2d\tau
+\int^{t}_{0}D^{3}_{p}(\tau)\|(a,u)(\tau)\|^{\ell}_{\dot B^{-\sigma_{1}}_{2,\infty}}d\tau,
\end{multline}
where $$D^{1}_{p}(t)\triangleq \|(a,u)\|^{\ell}_{\dot B^{\frac d2+1}_{2,1}}+\|a\|^{h}_{\dot B^{\frac dp}_{p,1}}+\|u\|^{h}_{\dot B^{\frac dp+1}_{p,1}},$$
$$D^{2}_{p}(t)\triangleq\|a\|^2_{\dot B^{\frac dp}_{p,1}},$$
$$D^{3}_{p}(t)\triangleq \Big(\|(a,u)\|^{\ell}_{\dot B^{\frac d2-1}_{2,1}}+\|a\|^{h}_{\dot B^{\frac dp}_{p,1}}+\|u\|^{h}_{\dot B^{\frac dp-1}_{p,1}}\Big)\Big(\|a\|^{h}_{\dot B^{\frac dp}_{p,1}}+\|u\|^{h}_{\dot B^{\frac dp+1}_{p,1}}\Big).$$
\end{lem}
\begin{proof}
It follows from \eqref{R-E23} that
\begin{eqnarray}\label{R-E56}
\frac 12 \frac{d}{dt}\|(a_k,v_k)\|^2_{L^2}+\|\Lambda v_k\|^2_{L^2}\leq (\|f_k\|_{L^2}+\|h_k\|_{L^2})\|(a_k,v_k)\|_{L^2}.
\end{eqnarray}
By performing a routine procedure, one can arrive at
\begin{multline}\label{R-E57}
\Big(\|(a,u)(t)\|^{\ell}_{\dot B^{-\sigma_{1}}_{2,\infty}}\Big)^2\lesssim \Big(\|(a_0,u_0)\|^{\ell}_{\dot B^{-\sigma_{1}}_{2,\infty}}\Big)^2+\int^{t}_{0}\|(f,g)\|^{\ell}_{\dot B^{-\sigma_{1}}_{2,\infty}}\|(a,u)(\tau)\|^{\ell}_{\dot B^{-\sigma_{1}}_{2,\infty}}d\tau.
\end{multline}
In what follow, we focus on the nonlinear norm $\|(f,g)\|^{\ell}_{\dot B^{-\sigma_{1}}_{2,\infty}}$. To this end, it is convenient to decompose $f$ and $g$ in terms of low-frequency and high-frequency parts:
\begin{eqnarray*}
f=f^{\ell}+f^{h}
\end{eqnarray*}
with
$$f^{\ell}\triangleq-a\,\mathrm{div}\,u^{\ell}-u\cdot\nabla a^{\ell},\ \ \ \ \ \ f^{h}\triangleq-a\,\mathrm{div}\,u^{h}-u\cdot\nabla a^{h}$$
and
\begin{eqnarray*}
g=g^{\ell}+g^{h}
\end{eqnarray*}
with
\begin{eqnarray*}
g^{\ell}\triangleq-u\cdot \nabla u^{\ell}-k(a) \nabla a^{\ell}+g_{3}(a,u^{\ell})+g_{4}(a,u^{\ell}),\\
g^{h}\triangleq-u\cdot \nabla u^{h}-k(a) \nabla a^{h}+g_{3}(a,u^{h})+g_{4}(a,u^{h}),
\end{eqnarray*}
where
\begin{eqnarray*}
&&g_{3}\left(a,v\right)=\frac {1}{1+a}\left(2\widetilde{\mu }\left(a\right)\mathrm{div}\,D\left(v\right)+
\widetilde{\lambda}(a) \nabla \mathrm{div}\,v\right)-I\left(a\right)\mathcal{A}v, \\
&&g_{4}(a,v)=\frac {1}{1+a}\left(2\widetilde{\mu}'(a)D\left(v\right)\cdot\nabla a+\widetilde{\lambda}'(a)\,\mathrm{div}\,v\,\nabla a\right)
\end{eqnarray*} and
$$z^{\ell}\triangleq\sum_{k\leq k_0} \dot{\Delta}_{k}z,\ \  \ \ z^{h}\triangleq z-z^{\ell} \ \ \hbox{for} \ \ z=a,u.$$

We use \eqref{R-E53}-\eqref{R-E54} to estimate those terms of $f$ and $g$ with $a^{\ell}$ or $u^{\ell}$. For instance, for
$a\,\mathrm{div}\,u^{\ell}=(a^{\ell}+a^{h})\mathrm{div}u^{\ell}$, one can get from \eqref{R-E53} that
\begin{eqnarray*}
\|a^{\ell}\,\mathrm{div}\,u^{\ell}\|^{\ell}_{\dot{B}^{-\sigma_{1}}_{2,\infty}}\lesssim \|\mathrm{div}\,u^{\ell}\|_{\dot{B}^{\frac dp}_{p,1}}\|a^{\ell}\|_{\dot{B}^{-\sigma_{1}}_{2,\infty}}\lesssim\|u\|^{\ell}_{\dot{B}^{\frac d2+1}_{2,1}}\|a\|^{\ell}_{\dot{B}^{-\sigma_{1}}_{2,\infty}}
\end{eqnarray*}
and
\begin{eqnarray*}
\|a^{h}\,\mathrm{div}\,u^{\ell}\|^{\ell}_{\dot{B}^{-\sigma_{1}}_{2,\infty}}\lesssim \|a^{h}\|_{\dot{B}^{\frac dp}_{p,1}}\|\mathrm{div}\,u^{\ell}\|_{\dot{B}^{-\sigma_{1}}_{2,\infty}}\lesssim \|a\|^{h}_{\dot{B}^{\frac dp}_{p,1}}\|u\|^{\ell}_{\dot{B}^{-\sigma_{1}}_{2,\infty}}.
\end{eqnarray*}
Therefore,
\begin{eqnarray}\label{R-E58}
\|a\,\mathrm{div}\,u^{\ell}\|_{\dot{B}^{-\sigma_{1}}_{2,\infty}}\lesssim \Big(\|u\|^{\ell}_{\dot{B}^{\frac d2+1}_{2,1}}+\|a\|^{h}_{\dot{B}^{\frac dp}_{p,1}}\Big)\|(a,u)\|^{\ell}_{\dot{B}^{-\sigma_{1}}_{2,\infty}}.
\end{eqnarray}

The estimates of $u\cdot\nabla a^{\ell}$ and $u\cdot \nabla u^{\ell}$ follow from essentially the same procedures as $a\,\mathrm{div}\,u^{\ell}$ so that
\begin{eqnarray}\label{R-E59}
\|u\cdot\nabla a^{\ell}\|^{\ell}_{\dot{B}^{-\sigma_{1}}_{2,\infty}}\lesssim \Big(\|a\|^{\ell}_{\dot{B}^{\frac d2+1}_{2,1}}+\|u\|^{h}_{\dot{B}^{\frac dp+1}_{p,1}}\Big)\|(a,u)\|^{\ell}_{\dot{B}^{-\sigma_{1}}_{2,\infty}} \end{eqnarray}
\begin{eqnarray}\label{R-E60}
\|u\cdot \nabla u^{\ell}\|^{\ell}_{\dot{B}^{-\sigma_{1}}_{2,\infty}}\lesssim \Big(\|u\|^{\ell}_{\dot{B}^{\frac d2+1}_{2,1}}+\|u\|^{h}_{\dot{B}^{\frac dp+1}_{p,1}}\Big)\|u\|^{\ell}_{\dot{B}^{-\sigma_{1}}_{2,\infty}}.
\end{eqnarray}
Bounding nonlinear terms involving composition functions is more elaborate. For example,
let us take a look at the term $k(a) \nabla a^{\ell}$. Keeping in mind that $k(0)=0$, one may write
$$k(a)=k'(0)a+\tilde{k}(a)a$$
for some smooth function $\tilde{k}$ vanishing at $0$. Arguing similarly as \eqref{R-E58} gives
\begin{eqnarray}\label{R-E61}
\|k'(0)a\nabla a^{\ell}\|^{\ell}_{\dot{B}^{-\sigma_{1}}_{2,\infty}} \lesssim \Big(\|a\|^{\ell}_{\dot{B}^{\frac d2+1}_{2,1}}+\|a\|^{h}_{\dot{B}^{\frac dp}_{p,1}}\Big)\|a\|^{\ell}_{\dot{B}^{-\sigma_{1}}_{2,\infty}}.
\end{eqnarray}
On the other hand, it follows from Propositions \ref{prop3.5}-\ref{prop3.6} that
\begin{eqnarray}\label{R-E62}
\|\tilde{k}(a)a\nabla a^{\ell}\|^{\ell}_{\dot{B}^{-\sigma_{1}}_{2,\infty}}\lesssim \|a\|^2_{\dot{B}^{\frac dp}_{p,1}}\|\nabla a^{\ell}\|_{\dot{B}^{-\sigma_{1}}_{2,\infty}}\lesssim \|a\|^2_{\dot{B}^{\frac dp}_{p,1}}\|a\|^{\ell}_{\dot{B}^{-\sigma_{1}}_{2,\infty}}.
\end{eqnarray}
Putting \eqref{R-E61} and \eqref{R-E62} together leads to
\begin{eqnarray}\label{R-E63}
\|k(a)\nabla a^{\ell}\|^{\ell}_{\dot{B}^{-\sigma_{1}}_{2,\infty}}\lesssim \Big(\|a\|^{\ell}_{\dot{B}^{\frac d2+1}_{2,1}}+\|a\|^{h}_{\dot{B}^{\frac dp}_{p,1}}+\|a\|^2_{\dot{B}^{\frac dp}_{p,1}}\Big)\|a\|^{\ell}_{\dot{B}^{-\sigma_{1}}_{2,\infty}}.
\end{eqnarray}
Similarly,
\begin{eqnarray}\label{R-E64}
\|g_{3}(a,u^{\ell})\|^{\ell}_{\dot{B}^{-\sigma_{1}}_{2,\infty}}\lesssim \Big(\|u\|^{\ell}_{\dot{B}^{\frac d2+1}_{2,1}}+\|a\|^{h}_{\dot{B}^{\frac dp}_{p,1}}+\|a\|^2_{\dot{B}^{\frac dp}_{p,1}}\Big)\|(a,u)\|^{\ell}_{\dot{B}^{-\sigma_{1}}_{2,\infty}}.
\end{eqnarray}
Next, we estimate  $g_{4}(a,u^{\ell})$. It suffices to estimate the first term in $g_{4}(a,u^{\ell})$, since the second one can be similarly handled. Denote by $J(a)$ the smooth function fulfilling
$J'(a)=\frac{2\mu'(a)}{1+a}$ and $J(0)=0$, so that $\nabla J(a)=\frac{2\mu'(a)}{1+a} \nabla a$. For convenience, we use the decomposition
 $J(a)=J'(0)a+\tilde{J}(a)a$. Then, it follows from \eqref{R-E53} that
\begin{eqnarray}\label{R-E65}
\|\nabla a^{\ell}\cdot D(u^{\ell})\|^{\ell}_{\dot{B}^{-\sigma_{1}}_{2,\infty}} \lesssim \|\nabla a^{\ell}\|_{\dot{B}^{\frac dp}_{p,1}}\| D(u^{\ell})\|_{\dot{B}^{-\sigma_{1}}_{2,\infty}} \lesssim \|a\|^{\ell}_{\dot{B}^{\frac d2+1}_{2,1}}\|u\|^{\ell}_{\dot{B}^{-\sigma_{1}}_{2,\infty}}.
\end{eqnarray}
Observe that the relation $\sigma_{1}\leq \sigma_{1}+\frac d2-\frac dp (p\geq2)$, we deduce that from \eqref{R-E54}
\begin{multline}\label{R-E66}
\|\nabla a^{h}\cdot D(u^{\ell})\|^{\ell}_{\dot{B}^{-\sigma_{1}}_{2,\infty}}\lesssim \|\nabla a^{h}\cdot D(u^{\ell})\|^{\ell}_{\dot{B}^{\frac dp-\frac d2-\sigma_{1}}_{2,\infty}}
\lesssim \|\nabla a^{h}\|_{\dot{B}^{\frac dp-1}_{p,1}}\|D(u^{\ell})\|_{\dot{B}^{\frac dp-\frac d2-\sigma_{1}+1}_{2,\infty}},
\end{multline}
Furthermore, we see that  $p\leq d^{*}$ is equivalent to $\frac dp-\frac d2-\sigma_{1}+1\geq-\sigma_{1}$, which implies that
\begin{eqnarray}\label{R-E67}
\|\nabla a^{h}\cdot D(u^{\ell})\|^{\ell}_{\dot{B}^{-\sigma_{1}}_{2,\infty}}\lesssim \|a\|^{h}_{\dot{B}^{\frac dp}_{p,1}}\|u\|^{\ell}_{\dot{B}^{-\sigma_{1}}_{2,\infty}}.
\end{eqnarray}
In addition, the remaining term with $\tilde{J}(a)a$ can be estimated as similarly:
\begin{multline}\label{R-E68}
\|\nabla (\tilde{J}(a)a)\cdot D(u^{\ell})\|^{\ell}_{\dot{B}^{-\sigma_{1}}_{2,\infty}}\lesssim \|\nabla (\tilde{J}(a)a)\cdot D(u^{\ell})\|^{\ell}_{\dot{B}^{\frac dp-\frac d2-\sigma_{1}}_{2,\infty}}\hfill \cr \hfill \lesssim \|\tilde{J}(a)a\|_{\dot{B}^{\frac dp}_{p,1}} \|D(u^{\ell})\|_{\dot{B}^{\frac dp-\frac d2-\sigma_{1}+1}_{2,\infty}}
\lesssim \|a\|^2_{\dot{B}^{\frac dp}_{p,1}}\|u\|^{\ell}_{\dot{B}^{-\sigma_{1}}_{2,\infty}}.
\end{multline}
Hence, together with \eqref{R-E65},\eqref{R-E67} and \eqref{R-E68}, we can conclude that
\begin{eqnarray}\label{R-E69}
\|g_{4}(a,u^{\ell})\|^{\ell}_{\dot{B}^{-\sigma_{1}}_{2,\infty}}\lesssim \Big(\|a\|^{\ell}_{\dot{B}^{\frac d2+1}_{2,1}}+\|a\|^{h}_{\dot{B}^{\frac dp}_{p,1}}+\|a\|^2_{\dot{B}^{\frac dp}_{p,1}}\Big)\|u\|^{\ell}_{\dot{B}^{-\sigma_{1}}_{2,\infty}}.
\end{eqnarray}

However, it seems to be difficult to get the suitable bounds for those terms of $f$ and $g$ with $a^{h}$ or $u^{h}$ by resorting to Proposition \ref{prop5.1} only. We will take advantage of the following result whose proof has been shown by \cite{DX}.
\begin{prop}\label{prop5.2} Let $k_0\in\Z,$ and denote $z^\ell\triangleq\dot S_{k_0}z,$  $z^h\triangleq z-z^\ell$ and, for any $s\in\R,$
$$
\|z\|_{\dot B^s_{2,\infty}}^\ell\triangleq\sup_{k\leq k_0}2^{ks} \|\ddk z\|_{L^2}.
$$
There exists a universal integer $N_0$ such that  for any $2\leq p\leq 4$ and $\sigma>0,$ we have
\begin{eqnarray}\label{R-E70}
&&\|f g^h\|_{\dot B^{-\sigma_0}_{2,\infty}}^\ell\leq C \bigl(\|f\|_{\dot B^\sigma_{p,1}}+\|\dot S_{k_0+N_0}f\|_{L^{p^*}}\bigr)\|g^h\|_{\dot B^{-\sigma}_{p,\infty}}\\\label{R-E71}
&&\|f^h g\|_{\dot B^{-\sigma_0}_{2,\infty}}^\ell
\leq C \bigl(\|f^h\|_{\dot B^\sigma_{p,1}}+\|\dot S_{k_0+N_0}f^h\|_{L^{p^*}}\bigr)\|g\|_{\dot B^{-\sigma}_{p,\infty}}
\end{eqnarray}
with  $\sigma_0\triangleq \frac{2d}p-\frac d2$ and $\frac1{p^*}\triangleq\frac12-\frac1p,$
and $C$ depending only on $k_0,$ $d$ and $\sigma.$
\end{prop}

Consider first the case $2\leq p\leq d$. If $2\leq p< d$, then \eqref{R-E70} with $\sigma=\frac dp-1$ yields
\begin{eqnarray}\label{R-E72}
\|f g^h\|_{\dot B^{-\sigma_1}_{2,\infty}}^\ell\lesssim \|f g^h\|_{\dot B^{-\sigma_0}_{2,\infty}}^\ell \lesssim \Big(\|f\|_{\dot B^{\frac dp-1}_{p,1}}+\|f^{\ell}\|_{L^{p^{*}}}\Big)\|g^h\|_{\dot B^{1-\frac dp}_{p,1}},
\end{eqnarray}
since $\sigma_1\leq \sigma_0$.
In the limit case $p=d$, one can get by the Sobolev embedding that
\begin{multline}\label{R-E73}
\|f g^h\|_{\dot B^{-\sigma_1}_{2,\infty}}^\ell\lesssim\|f g^h\|_{\dot B^{-\sigma_0}_{2,\infty}}^\ell\lesssim \|f g^h\|_{L^{\frac d2}}\leq \|f\|_{L^d}\| g^h\|_{L^d}\lesssim \|f\|_{\dot B^{0}_{d,1}}\|g^h\|_{\dot B^{0}_{d,1}}.
\end{multline}
Furthermore, using the embedding $\dot{B}^{\frac dp}_{2,1}\hookrightarrow L^{p^*}$ and the fact that $\frac d2-1\leq \frac dp$ and $1- \frac dp\leq\frac dp-1$, one has
\begin{eqnarray}\label{R-E74}
\|f g^h\|_{\dot B^{-\sigma_1}_{2,\infty}}^\ell \lesssim \Big(\|f^{\ell}\|_{\dot B^{\frac d2-1}_{2,1}}+\|f^h\|_{\dot B^{\frac dp-1}_{p,1}}\Big)\|g^h\|_{\dot B^{\frac dp-1}_{p,1}}.
\end{eqnarray}
Therefore, we get the following estimates:
\begin{eqnarray}\label{R-E75}
\|a\mathrm{div}u^{h}\|_{\dot B^{-\sigma_1}_{2,\infty}}^\ell \lesssim \Big(\|a\|^{\ell}_{\dot B^{\frac d2-1}_{2,1}}+\|a\|^h_{\dot B^{\frac dp}_{p,1}}\Big)\|u\|^h_{\dot B^{\frac dp+1}_{p,1}},\\\label{R-E76}
\|u\cdot\nabla a^h\|_{\dot B^{-\sigma_1}_{2,\infty}}^\ell \lesssim \Big(\|u\|^{\ell}_{\dot B^{\frac d2-1}_{2,1}}+\|u\|^h_{\dot B^{\frac dp-1}_{p,1}}\Big)\|a\|^h_{\dot B^{\frac dp}_{p,1}},\\\label{R-E77}
\|u\cdot \nabla u^h\|_{\dot B^{-\sigma_1}_{2,\infty}}^\ell \lesssim \Big(\|u\|^{\ell}_{\dot B^{\frac d2-1}_{2,1}}+\|u\|^h_{\dot B^{\frac dp-1}_{p,1}}\Big)\|u\|^h_{\dot B^{\frac dp+1}_{p,1}}.
\end{eqnarray}

Next, using the composition inequality and the embeddings $\dot{B}^{\frac dp}_{2,1}\hookrightarrow L^{p^*}$ and $\dot{B}^{\sigma_{0}}_{p,1}\hookrightarrow L^{p^*}$ yields
$$
 \|k(a)\|_{L^{p^*}}\lesssim  \|a\|_{L^{p^*}}\lesssim\|a^\ell\|_{\dot B^{\frac dp}_{2,1}}+\|a^h\|_{\dot B^{\sigma_0}_{p,1}}
 \lesssim \|a\|^\ell_{\dot B^{\frac d2-1}_{2,1}}+\|a\|^h_{\dot B^{\frac dp}_{p,1}}
 $$
and
$$
\|k(a)\|_{\dot B^{\frac dp-1}_{p,1}}\lesssim \|a\|_{\dot B^{\frac dp-1}_{p,1}}\lesssim \|a\|^\ell_{\dot B^{\frac d2-1}_{2,1}}+\|a\|^h_{\dot B^{\frac dp}_{p,1}}
$$
since $d/p-1>-d/p \ (p<2d)$. Therefore, we arrive at
\begin{eqnarray}\label{R-E78}
\|k(a)\nabla a^h\|_{\dot B^{-\sigma_1}_{2,\infty}}^\ell \lesssim \Big(\|a\|^{\ell}_{\dot B^{\frac d2-1}_{2,1}}+\|a\|^h_{\dot B^{\frac dp}_{p,1}}\Big)\|a\|^h_{\dot B^{\frac dp}_{p,1}}.
\end{eqnarray}
Similarly,
\begin{eqnarray}\label{R-E79}
\|g_{3}(a,u^h)\|_{\dot B^{-\sigma_1}_{2,\infty}}^\ell  \lesssim \Big(\|a\|^{\ell}_{\dot B^{\frac d2-1}_{2,1}}+\|a\|^h_{\dot B^{\frac dp}_{p,1}}\Big)\|u\|^h_{\dot B^{\frac dp+1}_{p,1}}.
\end{eqnarray}
Regarding the nonlinear term $g_4(a,u^h)$, the calculation is a little bit careful. Let us first deal with the case $\frac dp-\frac d2<\sigma_1\leq \sigma_0$ if $p\leq d$. By taking $f=\nabla K(a)$ and $g=\nabla u$, it follows from H\"{o}lder inequality that
\begin{multline}\label{R-E80}
 \|g_4(a,u^h)\|^{\ell}_{\dot{B}^{-\sigma_{1}}_{2,\infty}}\lesssim \|g_4(a,u^h)\|^{\ell}_{\dot{B}^{-\sigma_{0}}_{2,\infty}}= \|\nabla K(a)\otimes\nabla u^h\|^{\ell}_{\dot{B}^{-\sigma_{0}}_{2,\infty}}\lesssim \|\nabla K(a)\|_{L^p}\|\nabla u^h\|_{L^p}.
\end{multline}
The embeddings $\dot{B}^{\frac d2-\frac dp}_{2,1}\hookrightarrow L^p$ and $\dot{B}^{0}_{p,1}\hookrightarrow L^p$ yield that
\begin{eqnarray}\label{R-E81}
\|\nabla K(a)\|_{L^p}\lesssim \|\nabla a\|_{L^p} \lesssim \|\nabla a^{\ell}\|_{\dot{B}^{\frac d2-\frac dp}_{2,1}}+\|\nabla a^{h}\|_{\dot{B}^{\frac dp-1}_{p,1}}.
\end{eqnarray}
Owing to $-\sigma_{1}<\frac d2-\frac dp\leq \frac d2-1$, the real interpolation in Proposition \ref{prop3.3} and Young inequalities imply that
\begin{eqnarray}\label{R-E82}
\|\nabla a^{\ell}\|_{\dot{B}^{\frac d2-\frac dp}_{2,1}}\lesssim \|\nabla a^{\ell}\|^{1-\theta_{2}}_{\dot{B}^{-\sigma_{1}}_{2,\infty}}\|\nabla a^{\ell}\|^{\theta_{2}}_{\dot{B}^{\frac d2-1}_{2,\infty}}\lesssim \|a\|^{\ell}_{\dot{B}^{-\sigma_{1}}_{2,\infty}}+\|a\|^{\ell}_{\dot{B}^{\frac d2-1}_{2,1}},
\end{eqnarray}
where $$\theta_{2}=\frac{\sigma_{1}+\frac d2-\frac dp}{\sigma_{1}+\frac d2-1}\in (0,1].$$ Inserting \eqref{R-E81}-\eqref{R-E82} into \eqref{R-E80} leads to
\begin{eqnarray}\label{R-E83}
\|g_4(a,u^h)\|^{\ell}_{\dot{B}^{-\sigma_{1}}_{2,\infty}} \lesssim \Big(\|a\|^{\ell}_{\dot B^{\frac d2-1}_{2,1}}+\|a\|^h_{\dot B^{\frac dp}_{p,1}}\Big)\|u\|^h_{\dot B^{\frac dp+1}_{p,1}}+\|u\|^h_{\dot B^{\frac dp+1}_{p,1}}\|a\|^{\ell}_{\dot{B}^{-\sigma_{1}}_{2,\infty}}.
\end{eqnarray}
Next, we turn to the case $1-\frac d2<\sigma_1\leq \frac dp-\frac d2\leq 0$ if $2\leq p\leq d$. By Sobolev embedding properties and H\"{o}lder inequality, we arrive at
\begin{multline}\label{R-E833}
\|g_4(a,u^h)\|^{\ell}_{\dot{B}^{-\sigma_{1}}_{2,\infty}}\lesssim \|\nabla K(a)\otimes\nabla u^h\|^{\ell}_{\dot{B}^{0}_{2,\infty}} \lesssim \|\nabla K(a)\otimes\nabla u^h\|_{L^2}\lesssim \|\nabla K(a)\|_{L^d}\|\nabla u^h\|_{L^{d^*}},
\end{multline}
with $1/d+1/d^*=1/2$. It follows from Proposition \ref{prop3.4} and $\dot{B}^{\frac dp-1}_{p,1}\hookrightarrow L^d$ that
$$
\|\nabla K(a)\|_{L^d} \lesssim \|\nabla K(a)\|_{\dot{B}^{\frac dp-1}_{p,1}}\lesssim \|a\|_{\dot{B}^{\frac dp}_{p,1}}
$$
and
$$
\|\nabla u^h\|_{L^{d^*}}\lesssim \|\nabla u^h\|_{\dot{B}^{0}_{d^{*},1}}\lesssim \|\nabla u^h\|_{\dot{B}^{\frac{d}{d^*}}_{d^{*},1}}\lesssim \|u\|^{h}_{\dot{B}^{\frac{d}{p}+1}_{p,1}} \ \ (p\leq d^*).
$$
Therefore, we obtain
\begin{eqnarray}\label{R-E834}
\|g_4(a,u^h)\|^{\ell}_{\dot{B}^{-\sigma_{1}}_{2,\infty}} \lesssim \Big(\|a\|^{\ell}_{\dot B^{\frac d2-1}_{2,1}}+\|a\|^h_{\dot B^{\frac dp}_{p,1}}\Big)\|u\|^h_{\dot B^{\frac dp+1}_{p,1}}.
\end{eqnarray}

Consider now the oscillation case $p>d$. Again applying \eqref{R-E70} with $\sigma=1-\frac dp$ implies that
\begin{eqnarray}\label{R-E84}
\|f g^h\|_{\dot B^{-\sigma_1}_{2,\infty}}^\ell\lesssim \|f g^h\|_{\dot B^{-\sigma_0}_{2,\infty}}^\ell \lesssim \Big(\|f\|_{\dot B^{1-\frac dp}_{p,1}}+\|f^{\ell}\|_{L^{p^{*}}}\Big)\|g^h\|_{\dot B^{\frac dp-1}_{p,1}}.
\end{eqnarray}
By using the embedding $\dot B^{1-\sigma_0}_{2,1}\hookrightarrow \dot B^{1-\frac dp}_{p,1}$ and the fact $\frac d2-1<1-\sigma_0$ owing to $p>d$, we obtain
\begin{eqnarray}\label{R-E85}
\|f g^h\|_{\dot B^{-\sigma_1}_{2,\infty}}^\ell \lesssim \Big(\|f^{\ell}\|_{\dot B^{\frac d2-1}_{2,1}}+\|f^h\|_{\dot B^{\frac dp}_{p,1}}\Big)\|g^h\|_{\dot B^{\frac dp-1}_{p,1}}.
\end{eqnarray}
Consequently, we can get the same estimates for $a\mathrm{div}u^{h}, k(a)\nabla a^h$ and $g_{3}(a,u^h)$ as \eqref{R-E75}, \eqref{R-E78} and \eqref{R-E79}, respectively. In addition,
\begin{multline}\label{R-E86}
\|u\cdot\nabla a^h\|_{\dot B^{-\sigma_1}_{2,\infty}}^\ell \lesssim \Big(\|u\|^{\ell}_{\dot B^{\frac d2-1}_{2,1}}+\|u\|^h_{\dot B^{\frac dp}_{p,1}}\Big)\|a\|^h_{\dot B^{\frac dp}_{p,1}} \hfill \cr \hfill \lesssim \|u\|^{\ell}_{\dot B^{\frac d2-1}_{2,1}}\|a\|^h_{\dot B^{\frac dp}_{p,1}}+\|a\|^h_{\dot B^{\frac dp}_{p,1}}\|u\|^h_{\dot B^{\frac dp+1}_{p,1}}.
\end{multline}
On the other hand, the interpolation inequality implies that
\begin{eqnarray}\label{R-E87}
\|u^h\|^2_{\dot B^{\frac dp}_{p,1}}\lesssim \|u\|^h_{\dot B^{\frac dp-1}_{p,1}}\|u\|^h_{\dot B^{\frac dp+1}_{p,1}},
\end{eqnarray}
which leads to
\begin{multline}\label{R-E88}
\|u\cdot \nabla u^h\|_{\dot B^{-\sigma_1}_{2,\infty}}^\ell \lesssim \Big(\|u\|^{\ell}_{\dot B^{\frac d2-1}_{2,1}}+\|u\|^h_{\dot B^{\frac dp}_{p,1}}\Big)\|u\|^h_{\dot B^{\frac dp}_{p,1}} \lesssim \Big( \|u\|^{\ell}_{\dot B^{\frac d2-1}_{2,1}} +\|u\|^h_{\dot B^{\frac dp-1}_{p,1}}\Big)
\|u\|^h_{\dot B^{\frac dp+1}_{p,1}}
\end{multline}
To estimate $g_{4}(a,u^h)$, one can use \eqref{R-E71} with $\sigma=1-d/p$ to show that for any smooth
function $F$ vanishing at $0,$
  $$
  \|\nabla F(a)\otimes\nabla u^h\|^{\ell}_{\dot{B}^{-\sigma_{0}}_{2,\infty}}\lesssim
  \biggl(\|\nabla u^h\|_{\dot B^{1-\frac dp}_{p,1}}+\sum_{k=k_0}^{k_0+N_0-1}\|\ddk\nabla u^h\|_{L^{p^*}}\biggr)
  \|\nabla F(a)\|_{\dot B^{\frac dp-1}_{p,1}}.
  $$
  As $p^*\geq p,$ the Bernstein inequality ensures that $\|\ddk\nabla u^h\|_{L^{p^*}}\lesssim \|\ddk \nabla u^h\|_{L^p}$ for $k_0\leq k< k_0+N_0.$
  Hence, thanks to Proposition \ref{prop3.6}, the fact $\sigma_1\leq\sigma_0$ and $1-\frac dp<\frac dp,$ we have
 \begin{eqnarray}\label{R-E89}
 \hspace{5mm}
  \|g_4(a,u^h)\|^{\ell}_{\dot{B}^{-s_{1}}_{2,\infty}}\lesssim
  \|a\|_{\dot B^{\frac dp}_{p,1}}     \|\nabla u^h\|_{\dot B^{1-\frac dp}_{p,1}}
  \lesssim  \Big(\|a\|^{\ell}_{\dot B^{\frac d2-1}_{2,1}}+\|a\|^{h}_{\dot B^{\frac dp}_{p,1}}\Big) \|u\|^h_{\dot B^{\frac dp+1}_{p,1}}.
  \end{eqnarray}
By inserting above all estimates into \eqref{R-E57} yields \eqref{R-E55}.
\end{proof}

Noticing that the definition of $\cX_p(t)$ in Theorem \ref{thm1.1}, it is easy to see that
\begin{eqnarray}\label{R-E90}
\int^{t}_{0}\Big(D^{1}_{p}(\tau)+D^{3}_{p}(\tau)\Big)d\tau\leq \cX_p +\cX_p^{2}\leq C \cX_{p,0},\end{eqnarray}
since $\cX_{p,0}\ll 1$. In addition, one has
$$\|a^{\ell}\|^2_{L^2_{t}(\dot B^{\frac dp}_{p,1})}\lesssim \|a^{\ell}\|_{L^{\infty}_{t}(\dot B^{\frac dp-1}_{p,1})}\|a^{\ell}\|_{L^{1}_{t}(\dot B^{\frac dp+1}_{p,1})}\lesssim \|a\|^{\ell}_{L^{\infty}_{t}(\dot B^{\frac d2-1}_{2,1})}\|a\|^{\ell}_{L^{1}_{t}(\dot B^{\frac d2+1}_{2,1})}$$
and
$$
\|a^{h}\|^2_{L^2_{t}(\dot B^{\frac dp}_{p,1})}\lesssim \|a\|^{h}_{L^{\infty}_{t}(\dot B^{\frac dp}_{p,1})}\|a\|^{h}_{L^{1}_{t}(\dot B^{\frac dp}_{p,1})}.
$$
Consequently, it is shown that
\begin{eqnarray}\label{R-E91}
\int^{t}_{0}D^{2}_{p}(\tau)d\tau \lesssim\cX_p^{2}\leq C \cX_{p,0}.
\end{eqnarray}
Finally, combining \eqref{R-E90}-\eqref{R-E91}, one can employ nonlinear generalisations of the Gronwall's inequality (see for example, Page 360 of \cite{MPF}) and get
\begin{eqnarray}\label{R-E92}
\|(a,u)(t,\cdot)\|^{\ell}_{\dot B^{-\sigma_{1}}_{2,\infty}}\leq C_0
\end{eqnarray}
for all $t\geq0$, where $C_0>0$ depends on the norm $\|(a_0,u_0)\|^{\ell}_{\dot B^{-\sigma_{1}}_{2,\infty}}$.

%*******************************************************

\section{Proofs of main results}\setcounter{equation}{0}
The last section is devoted to the proofs of Theorem \ref{thm1.2} and Corollary \ref{cor1.1}.

\subsection{Proof of Theorem \ref{thm1.2}}
It follows from Lemmas \ref{lem4.1} and \ref{lem4.2} that
\begin{multline}\label{R-E93}
\frac{d}{dt}\Big(\|(a,u)^{\ell}\|_{\dot{B}^{\frac{d}{2}-1}_{2,1}}+\|(\nabla a, u)\|_{\dot B^{\frac dp-1}_{p,1}}^h\Big)+(\|(a,u)^{\ell}\|_{\dot{B}^{\frac{d}{2}+1}_{2,1}}+\|\nabla a\|_{\dot B^{\frac dp-1}_{p,1}}^h+\|u\|_{\dot B^{\frac dp+1}_{p,1}}^h)\hfill\cr\hfill \lesssim \|(f,g)\|^{\ell}_{\dot{B}^{\frac{d}{2}-1}_{2,1}}+\|f\|_{\dot B^{\frac dp-2}_{p,1}}^h+\|g\|_{\dot B^{\frac dp-1}_{p,1}}^h+\|\nabla u\|_{\dot B^{\frac dp}_{p,1}}\|a\|_{\dot B^{\frac dp}_{p,1}},
\end{multline}
where $$\|z\|^{\ell}_{\dot{B}^{s}_{2,1}}\triangleq \sum_{k\leq k_0}2^{ks}\|\dot{\Delta}_{k}z\|_{L^2}\quad \mbox{for} \ s\in \mathbb{R}.$$

Due to the fact $p\geq2$, the last term above can be bounded easily by $\cX_p(t)[\|u^{\ell}\|_{\dot{B}^{\frac{d}{2}+1}_{2,1}}+\|u\|^{h}_{\dot{B}^{\frac{d}{p}+1}_{p,1}}]$. Next, we have
$$\|f\|_{\dot B^{\frac dp-2}_{p,1}}^h\lesssim \|au\|^h_{\dot B^{\frac dp-1}_{p,1}}.$$
We decompose $au=a^{\ell}u^{\ell}+a^{\ell}u^{h}+a^{h}u$. Hence, it is shown that
$$\|a^{\ell}u^{h}\|^h_{\dot B^{\frac dp-1}_{p,1}}\lesssim \|a^{\ell}\|_{\dot B^{\frac dp-1}_{p,1}}\|u^{h}\|_{\dot B^{\frac dp}_{p,1}}\lesssim \cX_p(t) \|u\|^{h}_{\dot{B}^{\frac{d}{p}+1}_{p,1}}$$
and
$$\|a^{h}u\|^h_{\dot B^{\frac dp-1}_{p,1}}\lesssim\|a^{h}\|_{\dot B^{\frac dp}_{p,1}}\|u\|_{\dot B^{\frac dp-1}_{p,1}}\lesssim \cX_p(t) \|a\|^{h}_{\dot B^{\frac dp}_{p,1}}.$$
It follows from Proposition \ref{prop3.5} and Bernstein inequality that
$$
\|a^{\ell}u^{\ell}\|^h_{\dot B^{\frac dp-1}_{p,1}}\lesssim\|a^{\ell}u^{\ell}\|_{\dot B^{\frac d2+1}_{2,1}}\lesssim \|a^{\ell}\|_{L^\infty}\|u^{\ell}\|_{\dot B^{\frac d2+1}_{2,1}}+\|u^{\ell}\|_{L^\infty}\|a^{\ell}\|_{\dot B^{\frac d2+1}_{2,1}} \lesssim \cX_p(t) \|(a,u)^{\ell}\|_{\dot B^{\frac {d}{2}+1}_{2,1}}.
$$
Consequently,
\begin{eqnarray}\label{R-E94}
\|f\|_{\dot B^{\frac dp-2}_{p,1}}^h\lesssim  \cX_p(t)\Big(\|a\|^{h}_{\dot B^{\frac dp}_{p,1}}+\|(a,u)^{\ell}\|_{\dot B^{\frac {d}{2}+1}_{2,1}}+\|u\|^{h}_{\dot{B}^{\frac{d}{p}+1}_{p,1}}\Big).\end{eqnarray}
In addition, it is easy to get
\begin{eqnarray}\label{R-E95}
\|g\|_{\dot B^{\frac dp-1}_{p,1}}^h\lesssim \|u\|_{\dot B^{\frac dp-1}_{p,1}} \|\nabla u\|_{\dot B^{\frac dp}_{p,1}}+\|a\|_{\dot B^{\frac dp}_{p,1}} \|\nabla u\|_{\dot B^{\frac dp}_{p,1}}+\|a\|^2_{\dot B^{\frac dp}_{p,1}},
\end{eqnarray}
where the last term can be estimated as
$$\|a^{h}\|^2_{\dot B^{\frac dp}_{p,1}}\lesssim \cX_p(t) \|a\|^{h}_{\dot B^{\frac dp}_{p,1}},$$
$$\|a^{\ell}\|^2_{\dot B^{\frac dp}_{p,1}}\lesssim \|a^{\ell}\|_{\dot B^{\frac dp-1}_{p,1}}\|a^{\ell}\|_{\dot B^{\frac dp+1}_{p,1}}\lesssim \cX_p(t) \|a^{\ell}\|_{\dot B^{\frac {d}{2}+1}_{2,1}}.$$
Therefore, we arrive at
\begin{eqnarray}\label{R-E955a}
\|g\|_{\dot B^{\frac dp-1}_{p,1}}^h\lesssim\cX_p(t)\Big(\|a\|^{h}_{\dot B^{\frac dp}_{p,1}}+\|(a,u)^{\ell}\|_{\dot B^{\frac {d}{2}+1}_{2,1}}+\|u\|^{h}_{\dot{B}^{\frac{d}{p}+1}_{p,1}}\Big).
\end{eqnarray}

Bounding $\|(f,g)\|^{\ell}_{\dot{B}^{\frac{d}{2}-1}_{2,1}}$ is a little bit more complicated. We claim that
\begin{eqnarray}\label{R-E96}
\|(f,g)\|^{\ell}_{\dot{B}^{\frac{d}{2}-1}_{2,1}}\lesssim  \cX_p(t) \Big(\|a\|^{h}_{\dot B^{\frac dp}_{p,1}}+\|(a,u)^{\ell}\|_{\dot B^{\frac {d}{2}+1}_{2,1}}+\|u\|^{h}_{\dot{B}^{\frac{d}{p}+1}_{p,1}}\Big).
\end{eqnarray}
Here, we shall follow the similar strategy as in \cite{DR2} and handle the nonlinear term $I(a)\mathcal{A}u$ as example. To this end, we need the following two inequalities (see \cite{DR2}):
\begin{multline}\label{R-E955b}
\|T_{f}g\|_{\dot{B}^{s-1+\frac d2-\frac dp}_{p,1}}\lesssim \|f\|_{\dot{B}^{\frac dp-1}_{p,1}}\|g\|_{\dot{B}^{s}_{p,1}}\quad \mbox{if}\ d\geq2 \ \mbox{and}\ \frac{d}{d-1}\leq p \leq\min (4, d^{*}),
\end{multline}
\begin{multline}\label{R-E955c}
\|R(f,g)\|_{\dot{B}^{s-1+\frac d2-\frac dp}_{p,1}}\lesssim \|f\|_{\dot{B}^{\frac dp-1}_{p,1}}\|g\|_{\dot{B}^{s}_{p,1}} \quad \mbox{if} \ s>1-\min \Big(\frac dp,\frac {d}{p'}\Big)\ \mbox{and}\  1\leq p \leq4,
\end{multline}
where $1/p+1/p'=1$ and $d^{*}\triangleq \frac{2d}{d-2}$. Now, using Bony's para-product decomposition, one has that
$$I(a)\mathcal{A}u=T_{\mathcal{A}u}I(a)+R(I(a),\mathcal{A}u)+T_{I(a)}\mathcal{A}u^{\ell}+T_{I(a)}\mathcal{A}u^{h}.$$
Thanks to \eqref{R-E955b} and \eqref{R-E955c} with $s=\frac dp$, one can get
$$
\|T_{\mathcal{A}u}I(a)\|^{\ell}_{\dot{B}^{\frac d2-1}_{2,1}} \lesssim \|\mathcal{A}u\|_{\dot{B}^{\frac dp-1}_{p,1}}\|I(a)\|_{\dot{B}^{\frac dp}_{p,1}}  \lesssim \|a\|_{\dot{B}^{\frac dp}_{p,1}} \|\nabla u\|_{\dot{B}^{\frac dp}_{p,1}},
$$
$$
\|R(I(a),\mathcal{A}u)\|^{\ell}_{\dot{B}^{\frac d2-1}_{2,1}} \lesssim  \|a\|_{\dot{B}^{\frac dp}_{p,1}} \|\nabla u\|_{\dot{B}^{\frac dp}_{p,1}}.
$$
Observe that the right-side norm of above two inequalities can be bounded by
\begin{eqnarray*}
\|a\|_{\dot{B}^{\frac dp}_{p,1}} \|\nabla u\|_{\dot{B}^{\frac dp}_{p,1}}\lesssim \cX_p(t) \Big(\|u^{\ell}\|_{\dot B^{\frac {d}{2}+1}_{2,1}}+\|u\|^{h}_{\dot{B}^{\frac{d}{p}+1}_{p,1}}\Big).
\end{eqnarray*}
Since $T$ maps $L^\infty\times \dot{B}^{\frac d2-1}_{2,1}$ to $\dot{B}^{\frac d2-1}_{2,1}$, thus,
$$
\|T_{I(a)}\mathcal{A}u^{\ell}\|^{\ell}_{\dot{B}^{\frac d2-1}_{2,1}} \lesssim \|I(a)\|_{L^\infty} \|\mathcal{A}u^{\ell}\|_{\dot{B}^{\frac d2-1}_{2,1}}
\lesssim  \|a\|_{\dot{B}^{\frac dp}_{p,1}} \|u^{\ell}\|_{\dot{B}^{\frac {d}{2}+1}_{2,1}}.
$$
In order to handle the last term in the decomposition of $I(a)\mathcal{A}u$, we observe that owing to the spectral cut-off, there exists a universal integer $N_0$ such that
$$\Big(T_{I(a)}\mathcal{A}u^{h}\Big)^{\ell}=\dot{S}_{k_{0}+1}\Big(\sum_{|j-k_0|\leq N_{0}}\dot{S}_{j-1}I(a)\dot{\Delta}_{j}\mathcal{A}u^{h}\Big).$$
Hence $\|T_{I(a)}\mathcal{A}u^{h}\|^{\ell}_{\dot{B}^{\frac d2-1}_{2,1}}\approx 2^{k_0(\frac d2-1)}\sum_{|j-k_0|\leq N_{0}}\|\dot{S}_{j-1}I(a)\dot{\Delta}_{j}\mathcal{A}u^{h}\|_{L^2}$. If $2\leq p \leq \min (d,d^*)$ then one may use for $|j-k_0|\leq N_{0}$
\begin{multline*}
2^{k_0(\frac d2-1)}\|\dot{S}_{j-1}I(a)\dot{\Delta}_{j}\mathcal{A}u^{h}\|_{L^2}\lesssim \|\dot{S}_{j-1}I(a)\|_{L^d}\Big(2^{j(\frac{d}{d^*}-1)}\|\dot{\Delta}_{j}\mathcal{A}u^{h}\|_{L^{d^{*}}}\Big) \cr \lesssim \|a\|_{\dot{B}^{\frac dp-1}_{p,1}}\| u\|^{h}_{\dot{B}^{\frac dp+1}_{p,1}}\lesssim \cX_p(t)\| u\|^{h}_{\dot{B}^{\frac dp+1}_{p,1}},
\end{multline*}
and if $d\leq p\leq 4$, then it holds that
\begin{eqnarray*}
2^{k_0(\frac d2-1)}\|\dot{S}_{j-1}I(a)\dot{\Delta}_{j}\mathcal{A}u^{h}\|_{L^2}&\lesssim& \Big(2^{j\frac d4}\|\dot{S}_{j-1}I(a)\|_{L^4}\Big)\Big(2^{j(\frac d4-1)}\|\dot{\Delta}_{j}\mathcal{A}u^{h}\|_{L^4}\Big)\\ &\lesssim& 2^{k_0}\Big(2^{j(\frac dp-1)}\|\dot{S}_{j-1}I(a)\|_{L^p}\Big)\Big(2^{j(\frac dp-1)}\|\dot{\Delta}_{j}\mathcal{A}u^{h}\|_{L^p}\Big) \\ &\lesssim&\|a\|_{\dot{B}^{\frac dp-1}_{p,1}}\| u\|^{h}_{\dot{B}^{\frac dp+1}_{p,1}}\lesssim \cX_p(t)\| u\|^{h}_{\dot{B}^{\frac dp+1}_{p,1}}.
\end{eqnarray*}
Bounding other nonlinear terms is similar, so the details are left to the interested reader. Hence, the inequality \eqref{R-E96} is proved.

Inserting above estimates into (\ref{R-E93}) and using the fact that  $\cX_p(t)\lesssim  \cX_{p,0}\ll 1$ for all $t\geq0$ guaranteed  by Theorem \ref{thm1.1},
we conclude that
\begin{multline}\label{R-E97}
\frac{d}{dt}\Big(\|(a,u)^{\ell}\|_{\dot{B}^{\frac{d}{2}-1}_{2,1}}+\|(\nabla a, u)\|_{\dot B^{\frac dp-1}_{p,1}}^h\Big)+(\|(a,u)^{\ell}\|_{\dot{B}^{\frac{d}{2}+1}_{2,1}}+\|a\|_{\dot B^{\frac dp}_{p,1}}^h+\|u\|_{\dot B^{\frac dp+1}_{p,1}}^h)\leq0.
\end{multline}

%In what follows, we prove the interpolation inequalities that hold for those restricted Besov norms.
%\begin{prop}\label{prop6.1}
%Suppose that $m\neq \rho$. Then it holds that
%\begin{eqnarray*}
%\|f\|^{\ell}_{\dot{B}^{j}_{p,1}}\lesssim (\|f\|^{\ell}_{\dot{B}^{m}_{r,\infty}})^{1-\theta} (\|f\|^{\ell}_{\dot{B}^{\rho}_{r,\infty}})^{\theta}, \ \ \
%\|f\|^{h}_{\dot{B}^{j}_{p,1}}\lesssim (\|f\|^{h}_{\dot{B}^{m}_{r,\infty}})^{1-\theta} (\|f\|^{h}_{\dot{B}^{\rho}_{r,\infty}})^{\theta},
%\end{eqnarray*}
%where $j+d(\frac 1r-\frac 1p)=m(1-\theta)+\rho\theta$ for $0<\theta<1$ and $1\leq r \leq p\leq\infty$.
%\end{prop}
%\begin{proof}
%Without loss of generality, one may suppose that $m<\rho$. For $R\in \mathbb{R}$ to be chosen later, we have
%\begin{multline*}
%\|f\|^{\ell}_{\dot{B}^{j}_{p,1}}=\sum_{k\leq k_0}2^{kj}\|\dot{\Delta}_{k}f\|_{L^p}\lesssim \sum_{k\leq k_0}2^{kj+kd(\frac 1r-\frac 1p)}\|\dot{\Delta}_{k}f\|_{L^r}\hfill\cr\hfill\lesssim 2^{(j+d(\frac 1r-\frac 1p)-m)k_0}\|f\|^{\ell}_{\dot{B}^{m}_{r,\infty}}\lesssim 2^{(j+d(\frac 1r-\frac 1p)-m)R}\|f\|^{\ell}_{\dot{B}^{m}_{r,\infty}}.
%\end{multline*}
%Choosing
%$$R=\log_{2}\Big(\frac{\|f\|^{\ell}_{\dot{B}^{\rho}_{r,\infty}}}{\|f\|^{\ell}_{\dot{B}^{m}_{r,\infty}}}\Big)^{\frac{1}{\rho-m}}$$
%yields the first inequality. The proof of the second inequality is left to the interested reader.
%\end{proof}
In what follow, we shall see that interpolation  plays the key role to get the time-decay estimates. Thanks to $-\sigma_{1}<\frac d2-1\leq \frac dp<\frac d2+1$, it follows from the real interpolation in Proposition \ref{prop3.3} that
\begin{eqnarray}\label{R-E98}
\|(a,u)^{\ell}\|_{\dot{B}^{\frac{d}{2}-1}_{2,1}}\lesssim \Big(\|(a,u)\|^{\ell}_{\dot{B}^{-\sigma_{1}}_{2,\infty}}\Big)^{\theta_{0}}\Big(\|(a,u)^{\ell}\|_{\dot{B}^{\frac{d}{2}+1}_{2,\infty}}\Big)^{1-\theta_{0}},
\end{eqnarray}
where $\theta_{0}=\frac{2}{d/2+1+\sigma_1}\in (0,1)$.
%Similarly, we also have
%\begin{eqnarray}
%\|u\|_{\dot B^{\frac dp-1}_{p,1}}^h\lesssim \Big(\|u\|^{h}_{\dot{B}^{-d/p}_{p,1}}\Big)^{\theta_{0}}\Big(\|u\|_{\dot B^{\frac dp+1}_{p,1}}^h\Big)^{1-\theta_{0}}.
%\end{eqnarray}

%On the other hand, we see that
%\begin{eqnarray}
%\|(a,u)\|^{\ell}_{\dot{B}^{-\sigma_0}_{2,1}}+\|(a,u)\|^{h}_{\dot{B}^{-d/p}_{p,1}}\lesssim \|(a,u)\|_{\dot{B}^{-\sigma_0}_{2,1}}.
%\end{eqnarray}
%In next section, we shall bound the evolution of the negative Besov norm of solutions, say
%$$\|(a,u)\|^{\ell}_{\dot{B}^{-\sigma_0}_{2,1}}\leq C_0$$
%for all $t\geq0$, where $C_0>0$ is relative to the norm $\|(a_0,u_0)\|^{\ell}_{\dot{B}^{-\sigma_0}_{2,1}}$.
By virtue of \eqref{R-E92}, one can get
$$
\|(a,u)^{\ell}\|_{\dot{B}^{\frac{d}{2}+1}_{2,1}}\geq c_0 \Big(\|(a,u)^{\ell}\|_{\dot{B}^{\frac{d}{2}-1}_{2,1}}\Big)^{\frac{1}{1-\theta_{0}}},$$
where $c_0=C^{-\frac{1}{1-\theta_{0}}}C_{0}^{-\frac{\theta_{0}}{1-\theta_{0}}}$. In addition, it follows the fact
$\|(\nabla a,u)|_{\dot B^{\frac dp-1}_{p,1}}^h\leq\cX_p(t)\lesssim  \cX_{p,0}\ll 1$ for all $t\geq0$ that
$$\|a\|_{\dot B^{\frac dp}_{p,1}}^h\geq \Big(\|a\|_{\dot B^{\frac dp}_{p,1}}^h\Big)^{\frac{1}{1-\theta_{0}}},\quad
\|u\|_{\dot B^{\frac dp+1}_{p,1}}^h\geq \Big(\|u\|_{\dot B^{\frac dp-1}_{p,1}}^h\Big)^{\frac{1}{1-\theta_{0}}}.$$
Consequently, there exists a constant $\tilde{c}_{0}>0$ such that the following Lyapunov-type inequality holds
\begin{multline}\label{R-E99}
\frac{d}{dt}\Big(\|(a,u)^{\ell}\|_{\dot{B}^{\frac{d}{2}-1}_{2,1}}+\|(\nabla a, u)\|_{\dot B^{\frac dp-1}_{p,1}}^h\Big)
+\tilde{c}_{0}\Big(\|(a,u)^{\ell}\|_{\dot{B}^{\frac{d}{2}-1}_{2,1}}+\|(\nabla a, u)\|_{\dot B^{\frac dp-1}_{p,1}}^h\Big)^{1+\frac{2}{d/2-1+\sigma_1}}\leq 0.
\end{multline}
Solving \eqref{R-E99} directly gives
\begin{multline}\label{R-E100}
\|(a,u)^{\ell}(t)\|_{\dot{B}^{\frac{d}{2}-1}_{2,1}}+\|(\nabla a, u)(t)\|_{\dot B^{\frac dp-1}_{p,1}}^h\leq \Big(\cX_{p,0}^{-\frac{2}{d/2-1+\sigma_1}}+\frac{2\tilde{c}_{0}t}{d/2-1+\sigma_1}\Big)^{-\frac{d/2-1+\sigma_1}{2}} \hfill \cr\hfill \lesssim (1+t)^{-\frac{d/2-1+\sigma_1}{2}}
\end{multline}
for all $t\geq 0$. According to the embedding properties in Proposition \ref{prop3.4}, we arrive at
\begin{multline}\label{R-E101}
\|(a,u)(t)\|_{\dot{B}^{\frac{d}{p}-1}_{p,1}}\lesssim \|(a,u)^{\ell}(t)\|_{\dot{B}^{\frac{d}{2}-1}_{2,1}}+\|(\nabla a, u)(t)\|_{\dot B^{\frac dp-1}_{p,1}}^h \lesssim (1+t)^{-\frac{d/2-1+\sigma_1}{2}}.
\end{multline}
In addition, if $\sigma \in (-\tilde{\sigma}_{1}, \frac dp-1)$ with $\tilde{\sigma}_{1}=\sigma_{1}+d(\frac 12-\frac1p)$, then employing Proposition \ref{prop3.3} once again implies that
\begin{multline}\label{R-E102}
\|(a,u)^{\ell}\|_{\dot{B}^{\sigma}_{p,1}}\lesssim \|(a,u)^{\ell}\|_{\dot{B}^{\sigma+d(\frac 12-\frac 1p)}_{2,1}}\lesssim \Big(\|(a,u)\|^{\ell}_{\dot{B}^{-\sigma_1}_{2,\infty}}\Big)^{\theta_1}\Big(\|(a,u)^{\ell}\|_{\dot{B}^{\frac d2-1}_{2,\infty}}\Big)^{1-\theta_1},
\end{multline}
where $$\theta_1=\frac{\frac dp -1-\sigma}{\frac d2-1+\sigma_1}\in (0,1).$$ Noticing the fact that
$$\|(a,u)(t)\|^{\ell}_{\dot{B}^{-\sigma_{1}}_{2,\infty}} \leq C_{0}$$
for all $t\geq 0$, with aid of \eqref{R-E101}-\eqref{R-E102}, we deduce that
\begin{eqnarray}\label{R-E103}
\|(a,u)(t)^{\ell}\|_{\dot{B}^{\sigma}_{p,1}}\lesssim \Big[(1+t)^{-\frac{d/2-1+\sigma_1}{2}}\Big]^{1-\theta_1}=(1+t)^{-\frac{d}{2}(\frac 12-\frac 1p)-\frac{\sigma+\sigma_1}{2}}
\end{eqnarray}
for all $t\geq 0$, which leads to
\begin{multline}\label{R-E104}
\|(a,u)(t)\|_{\dot{B}^{\sigma}_{p,1}}\lesssim \|(a,u)(t)^{\ell}\|_{\dot{B}^{\sigma}_{p,1}}+\|(a,u)(t)\|^{h}_{\dot{B}^{\sigma}_{p,1}}\lesssim(1+t)^{-\frac{d}{2}(\frac 12-\frac 1p)-\frac{\sigma+\sigma_1}{2}}
\end{multline}
for $\sigma \in (-\tilde{\sigma}_{1}, \frac dp-1)$. Therefore, the proof of Theorem \ref{thm1.2} is completed by $\dot{B}^{0}_{p,1}\hookrightarrow L^p.$
$\Box$

\subsection{Proof of Corollary \ref{cor1.1}}
Indeed, Corollary \ref{cor1.1} can be regarded as the direct consequence
of Proposition \ref{prop3.8}. It follows from Proposition \ref{prop3.8} with  $q=p,$ $m=\frac dp-1$
and $k=-\tilde{\sigma}_{1}+\ep$ with  $\ep>0$ small enough. Furthermore,
if we define  $\theta_{2}$ by the relation
$$
k\theta_{2}+m(1-\theta_{2})=l+d\Bigl(\frac1p-\frac1r\Bigr),
$$
then one can take  $\ep$ so small as $\theta_{2}>0$ to be  in $(0,1).$ Therefore we have
\begin{eqnarray}\label{R-E105}
\|\Lambda^{l}(a,u)\|_{L^{r}} &\lesssim &\|\Lambda^{m}(a,u)\|_{L^p}^{1-\theta_{2}}\|\Lambda^{k}(a,u)\|^{\theta_{2}}_{L^p}
\nonumber\\& \lesssim & \Big\{(1+t)^{-\frac d2(\frac 12-\frac 1p)-\frac{m+\sigma_{1}}{2}}\Big\}^{1-\theta_{2}}
\Big\{(1+t)^{-\frac d2(\frac 12-\frac 1p)-\frac{k+\sigma_{1}}{2}}\Big\}^{\theta_{2}}
\nonumber\\&=& (1+t)^{-\frac d2(\frac 12-\frac 1r)-\frac {l+\sigma_{1}}{2}}
\end{eqnarray}
for $p\leq r\leq\infty$ and $l\in\R$ satisfying
$-\tilde{\sigma}_1<l+d\big(\frac1p-\frac1r\big) \leq\frac dp-1\cdotp$ \quad \quad $\Box$

\bigskip

{\bf Acknowledgments}
The first author (Z.~Xin) is supported in part by Hong Kong Research Council Earmarked Grants CUHK 14305315, CUHK 14300917 and CUHK 14302917,
NSFC-RGC Joint Grant N-CUHK 443-14 and Zheng Ge Ru Foundation.
The second author (J. Xu) is partially supported by the National
Natural Science Foundation of China (11471158, 11871274) and the Fundamental Research Funds for the Central
Universities (NE2015005).

\end{document}